\documentclass[12pt]{amsart}
\usepackage[all]{xy}
\usepackage{a4wide}
\usepackage{amssymb}
\usepackage{amsthm}
\usepackage{amsmath}
\usepackage{amscd}
\usepackage{verbatim}
\usepackage{color}
\usepackage[all]{xy}
\usepackage[OT2,T1]{fontenc}
\DeclareSymbolFont{cyrletters}{OT2}{wncyr}{m}{n}
\DeclareMathSymbol{\Sha}{\mathalpha}{cyrletters}{"58}
\textheight8.7in \textwidth6.75in \numberwithin{equation}{section}

\theoremstyle{plain}
\newtheorem{theorem}{Theorem}[section]

\newtheorem{lemma}[theorem]{Lemma}
\newtheorem{proposition}[theorem]{Proposition}

\theoremstyle{definition}

\newtheorem*{example}{Example}

\theoremstyle{remark}
\newtheorem{remark}{Remark}

\newcommand{\R}{\mathbb{R}}
\newcommand{\Q}{\mathbb{Q}}
\newcommand{\Z}{\mathbb{Z}}
\newcommand{\N}{\mathbb{N}}
\newcommand{\C}{\mathbb{C}}
\newcommand{\h}{\mathbb{H}}
\renewcommand{\H}{\mathbb{H}}

\newcommand{\whZ}{\widehat{\mathfrak{Z}}}
\newcommand{\ord}{{\text {\rm ord}}}
\newcommand{\z}{\mathfrak{z}}

\newcommand{\leg}[2]{\left( \frac{#1}{#2} \right)}

\newcommand{\rk}{{\text {\rm rk}}}

\newcommand{\calE}{\mathcal{E}}
\newcommand{\calI}{\mathcal{I}}
\newcommand{\calF}{\mathcal{F}}

\newcommand{\calM}{\mathcal{M}}

\newcommand{\calZ}{\mathfrak{Z}}

\newcommand{\frake}{\mathfrak e}

\newcommand{\eps}{\varepsilon}

\newcommand{\tr}{\operatorname{tr}}

\newcommand{\sgn}{\operatorname{sgn}}

\newcommand{\Sl}{\operatorname{SL}}

\newcommand{\Mp}{\operatorname{Mp}}
\newcommand{\Orth}{\operatorname{O}}

\newcommand{\thetaL}[1]{\Theta_{\Delta,r}(\tau,z,#1)}
\newcommand{\phish}{\varphi_{\text{Sh}}}

\newcommand{\SL}{{\text {\rm SL}}}
\newcommand{\G}{\Gamma}
\newcommand{\dg}{\mathcal{D}} %discriminant group
\newcommand{\dgdelta}{{\mathcal{D}(\Delta)}}
\newcommand{\abs}[1]{\left\vert#1\right\vert}
\newcommand{\e}{\mathfrak{e}}
\newcommand{\smallabcd}{\left(\begin{smallmatrix}a & b \\ c & d\end{smallmatrix}\right)}
\newcommand{\Deltaover}[1]{\left(\frac{\Delta}{#1}\right)}

\newcommand{\smallTmatrix}{\left(\begin{smallmatrix}1 & 1 \\ 0 & 1\end{smallmatrix}\right)}
\newcommand{\smallSmatrix}{\left(\begin{smallmatrix}0 & -1 \\ 1 & 0\end{smallmatrix}\right)}
\newcommand{\M}{\mathcal{M}}

\begin{document}

\title[Weierstrass mock modular forms and elliptic curves]{Weierstrass mock modular forms and
elliptic curves}

\author{Claudia Alfes, Michael Griffin, Ken Ono, and Larry Rolen}

\address{Fachbereich Mathematik,
Technische Universit\"at Darmstadt, Schlossgartenstrasse 7, 64289
Darmstadt, Germany} \email{alfes@mathematik.tu-darmstadt.de}

\address{Department of Mathematics and Computer Science, Emory University,
Atlanta, Georgia 30022}\email{mjgrif3@emory.edu} \email{ono@mathcs.emory.edu}

\address{Mathematical Institute, University of Cologne, Weyertal 86-90, 50931 Cologne,
Germany}\email{lrolen@mi.uni-koeln.de}

\thanks{The first author is supported by the DFG Research Unit FOR 1920
"Symmetry, Geometry and Arithmetic". The second three authors thank the generous support of the National
Science Foundation, and the third author also thanks the
Asa Griggs Candler Fund. The fourth author thanks the University of Cologne and the DFG for their generous support via the University of Cologne postdoc grant DFG Grant D-72133-G-403-151001011, funded under the Institutional Strategy of the University of Cologne within the German Excellence Initiative.}

\subjclass[2010]{11F37, 11G40, 11G05, 11F67}

\begin{abstract}{\it Mock modular forms}, which give the theoretical framework for
Ramanujan's enigmatic mock theta functions, play many roles in
mathematics.
We study their role in the context of modular parameterizations of elliptic curves $E/\Q$.
We show that mock modular forms
which arise from Weierstrass $\zeta$-functions
encode the central $L$-values and $L$-derivatives which occur in the Birch and Swinnerton-Dyer Conjecture. By defining a theta lift using a  kernel recently studied by H\"ovel,
we obtain
canonical weight 1/2 harmonic Maass forms whose
Fourier coefficients encode the vanishing of these values for the quadratic twists of $E$.
We employ results of
Bruinier and the third author, which builds on seminal work of Gross, Kohnen, Shimura,
Waldspurger, and Zagier. We also obtain $p$-adic formulas
for the corresponding weight 2 newform using the action of the Hecke algebra on the Weierstrass mock modular form.
\end{abstract}

\maketitle

\section{Introduction and Statement of Results}\label{sect:intro}

The theory of {\it mock modular forms}, which provides the underlying theoretical framework for
Ramanujan's enigmatic mock theta functions
\cite{BO1, BO2, Za3, Z2}, has recently played important roles in
combinatorics, number theory, mathematical physics, and representation theory
(see \cite{O2, O3, Za3}). Here we consider mock modular forms and the arithmetic
of elliptic curves.

We first recall the notion of a {\it harmonic weak Maass form} which was introduced by Bruinier and Funke \cite{BF}. Here we let $z:=x+iy\in \H$, where $x, y\in \R$, and we
let $q:=e^{2\pi i z}$. For an integer $N\geq 1$ we have the congruence subgroup $\G_0(N):=\left\{\left(\begin{smallmatrix}
                                                                                                                      a&b\\c&d
                                                                                                                     \end{smallmatrix}\right)\in \mathrm{SL}_2(\Z)\,:\,\,c\equiv 0\pmod{N}\right\}$.
We first recall the definition of the {\it Petersson slash operator} $|k$, which for given $\gamma=\left(\begin{smallmatrix}a&b\\ c&d\end{smallmatrix}\right)\in\operatorname{SL}_2(\Z)$ acts on functions $f$ by
\[
f|_k(\gamma)(z):=
(cz+d)^{-k}f\left(\frac{az+b}{cz+d}\right)
.
\] 
A {\it harmonic weak Maass form of weight $k\in \frac{1}{2}\Z$ on
$\Gamma_0(N)$} (with $4\vert N$ if $k\in \frac{1}{2}\Z\setminus \Z$)
is then a smooth function on $\H$, the upper-half of the complex plane,
which satisfies the following properties:
\begin{enumerate}
\item[(i)]
 $f\mid_k\gamma = f$ for all $\gamma\in \Gamma_0(N)$;
\item[(ii)] $\Delta_k f =0 $, where $\Delta_k$ is the weight $k$
hyperbolic Laplacian on $\H$ (see (\ref{deflap})); \item[(iii)]
There is a polynomial $P_f=\sum_{n\leq 0} c^+(n)q^n \in \C[q^{-1}]$
such that $$f(z)-P_f(z) = O(e^{-\eps y}),$$ as $y\to\infty$ for
some $\eps>0$. Analogous conditions are required at all
cusps.
\end{enumerate}

\begin{remark}
The polynomial $P_f$ is called the {\it principal part of $f$ at} $\infty$.
If $P_f$ is nonconstant, then $f$ has
exponential growth at the cusp $\infty$. Similar remarks apply at
all of the cusps.
\end{remark}

A weight $k$ harmonic Maass form\footnote{For convenience we shall refer to harmonic weak Maass forms as harmonic Maass forms.} $f(z)$ has a Fourier expansion of the form
\begin{equation}\label{fourier}
f(z)=f^{+}(z)+f^{-}(z)=\sum_{n\gg -\infty} c^+(n) q^n + \sum_{n<0} c^-(n)\G(1-k,4\pi |n|y) q^n,
\end{equation}
where $\Gamma(\alpha,x)$ is the incomplete
Gamma-function.  The function $f^{+}(z)=\sum_{n\gg -\infty}
c^+(n) q^n$ is the {\it holomorphic part} of $f(z)$, and its
complement $f^{-}(z)$ is its {\it nonholomorphic part}.
If $f^{-}=0$, then $f=f^{+}$ is a {\it weakly holomorphic modular form}.
If $f^{-}$ is nontrivial, then $f^{+}$ is called a {\it mock modular
form}.

Many recent applications of mock modular forms rely on the fact that weight $2-k$ harmonic Maass forms are intimately related to weight
$k$ modular forms by the
differential operator $$\xi_{2-k}:=-2i y^{2-k}\overline{\frac{\partial}{\partial \bar{z}}}.$$  Indeed, every weight $k$ cusp form $F$ is the image of infinitely many weight $2-k$ harmonic Maass forms
under $\xi_{2-k}$. Therefore, it is natural to seek ``canonical'' preimages.  Such a form should be readily  constructible from $F$,  and should
also encode deep underlying arithmetic information.

There is a canonical weight 0 harmonic Maass form which arises from the analytic realization of an elliptic curve $E/\Q$.
This was first observed by Guerzhoy \cite{Guerzhoy1, Guerzhoy2}.
To define it we recall that $E\cong \C/\Lambda_E$, where $\Lambda_E$ is a 2-dimensional lattice in $\C$. The parameterization of
$E$ is given by $\mathfrak{z}\mapsto P_{\z}=(\wp(\Lambda_E;\z),\wp'(\Lambda_E;\z))$, where
$$
\wp(\Lambda_E;\z):=
\frac{1}{\z^2}+\sum_{w\in \Lambda_E\setminus \{0\}}\left(
\frac{1}{(\z-w)^2}-\frac{1}{w^2}\right)
$$
is the usual Weierstrass $\wp$-function for $\Lambda_E$.
Here $E$ is given by the Weierstrass equation
$$
E\colon \ y^2=4x^3 - 60G_4(\Lambda_E)x-140G_6(\Lambda_E),
$$
where
$G_{2k}(\Lambda_E):=\sum_{w\in \Lambda_E\setminus \{0\}}w^{-2k}$
is the classical weight $2k$ Eisenstein series.
The canonical harmonic Maass form arises from the Weierstrass zeta-function
\begin{equation}\label{WeierstrassZeta}
\zeta(\Lambda_E;\z):=\frac{1}{\z}+\sum_{w\in \Lambda_E\setminus \{0\}}
\left( \frac{1}{\z-w}+\frac{1}{w}+\frac{z}{w^2}\right)=
\frac{1}{\z}-\sum_{k=1}^{\infty}G_{2k+2}(\Lambda_E)\z^{2k+1}.
\end{equation}
This function already plays important roles in the theory of elliptic curves.
The first role follows from 
the well-known ``addition law''
\begin{equation}\label{addlaw}
\zeta(\Lambda_E;\z_1+\z_2)=\zeta(\Lambda_E;\z_1)+\zeta(\Lambda_E;\z_2)+\frac{1}{2}
\frac{\wp'(\Lambda_E;\z_1)-\wp'(\Lambda_E;\z_2)}{\wp(\Lambda_E; \z_1)-\wp(\Lambda_E;\z_2)},
\end{equation}
which can be interpreted in terms of the ``group law''
of $E$.

To obtain the canonical forms from $\zeta(\Lambda_E;\z)$,
we make use of the modularity of elliptic curves over $\Q$, which gives the modular
parameterization
$$
\phi_E: X_0(N_E)\rightarrow \C/\Lambda_E \cong E,
$$
where $N_E$ is the conductor of $E$. For convenience, we suppose throughout that
$E$ is a strong Weil curve.
Let $F_E(z)=\sum_{n=1}^{\infty}
a_E(n)q^n\in S_2(\Gamma_0(N_E))$ be the associated newform, and let $\calE_E(z)$ be its {\it Eichler integral}
\begin{equation}\label{EichlerIntegral}
\calE_E(z):=-2\pi i \int_{z}^{i \infty}F_E(\tau) d\tau =\sum_{n=1}^{\infty}\frac{a_E(n)}{n}\cdot q^n.
\end{equation}
Motivated by an observation of Eisenstein (namely, that the modified $\zeta$ function given in \eqref{ZetaStar} is a non-holomorphic elliptic function), we define the function $\calZ^{+}_E(\z)$ by
\begin{equation}\label{MMF}
\calZ^{+}_E(\z):=\zeta(\Lambda_E;\z)-S(\Lambda_E)\z,
\end{equation}
where
\begin{equation}\label{S}
S(\Lambda_E):=\lim_{s\rightarrow 0^{+}}\sum_{w\in \Lambda_E \setminus \{0\}}
\frac{1}{w^2 |w|^{2s}}.
\end{equation}
We define the nonholomorphic function $\calZ_E(\z)$ by
\begin{equation}\label{CorrectedZeta}
\calZ_E(\z):=\calZ^{+}_E(\z)-\frac{\deg(\phi_E)}{4\pi ||F_E||^2}\cdot \overline{\z},
\end{equation}
where $||F_E||$ is the Petersson norm of $F_E$. Finally, we define the nonholomorphic function
$\whZ_{E}(z)$ on $\H$ by the specialization of this function at $\z=\calE_E(z)$ given by
\begin{equation}\label{HMF}
\whZ_E(z)=\whZ_E^{+}(z)+\whZ_E^{-}(z):=\calZ_E(\calE_E(z)).
\end{equation}
In particular, the holomorphic part of $\whZ_E(z)$ is  $\whZ_E^{+}(z)=\calZ_E^{+}(\calE_E(z))$.

\begin{theorem}\label{thm1}
Assume the notation and hypotheses above. The following are true:

\smallskip
\noindent
(1) The poles of $\whZ^{+}_E(z)$ are precisely those points $z$ for which
$\calE_E(z)\in \Lambda_E$.

\smallskip
\noindent
(2) If $\whZ^{+}_E(z)$ has poles in $\H$, then there is a canonical
modular function $M_E(z)$ with algebraic coefficients on $\Gamma_0(N_E)$ for which
$\whZ^{+}_E(z)-M_E(z)$ is holomorphic on $\H$.

\smallskip
\noindent
(3) We have that
$\whZ_E(z)-M_E(z)$ is a weight 0 harmonic Maass form on $\Gamma_0(N_E)$.
In particular, $\whZ^{+}_E(z)$ is a weight 0 mock modular form.
\end{theorem}

\begin{remark} 
Guerzhoy \cite{Guerzhoy1} has considered the construction of harmonic Maass forms using the Weierstrass $\zeta$ function in his
work on the Kaneko-Zagier hypergeometric differential equation, and in
\cite{Guerzhoy2} he studies their $p$-adic properties.
\end{remark}

\begin{remark} We refer to $\whZ^{+}_E(z)$ as the {\it Weierstrass mock modular form} for $E$.
It is a simple task to compute this mock modular form. Using the two  Eisenstein numbers
$G_4(\Lambda_E)$ and $G_6(\Lambda_E)$, one then computes the remaining Eisenstein numbers using the recursion
$$
G_{2n}(\Lambda_E):=\sum_{j=2}^{n-2}\frac{3(2j-1)(2n-2j-1)}{(2n+1)(2n-1)(n-3)}\cdot G_{2j}(\Lambda_E)G_{2n-2j}(\Lambda_E).
$$
Armed with the Fourier expansion of $F_E(z)$ and $S(\Lambda_E)$, one then simply applies (\ref{EichlerIntegral})-(\ref{HMF}).
\end{remark}

\begin{remark}
The number $\deg(\phi_E)$, which appears in (\ref{CorrectedZeta}),
gives information about modular form congruences. The {\it congruence number}
for $E$ is the largest integer, say $r_E$, with the property that there is a
$g\in S_2(\Gamma_0(N_E))\cap \Z[[q]]$, which is orthogonal to $F_E$
with respect to the Petersson inner product, which also satisfies $F_E\equiv g\pmod{r_E}$.
A theorem of Ribet asserts that $\deg(\phi_E)\mid r_E$ (see Theorem 2.2 of \cite{ARS}).
\end{remark}

Many applications require the explicit Fourier expansions of harmonic Maass forms at cusps (for example, knowing these expansion is useful for computing exact formulas for the coefficients and for computing theta lifts such as $\calI_{\Delta,r}(\bullet;z)$). The following theorem
gives such expansions for
the forms $\whZ_E(z)$
 in Theorem~\ref{thm1} at certain cusps. These expansions follow from the fact that these forms
  transform nicely under $\Gamma_0^*(N_E)$, the extension of $\Gamma_0(N_E)$ by the Atkin-Lehner involutions.
For each positive integer $q\vert N_E$ we define the Atkin-Lehner involution $W_q$ by any determinant $q^{\alpha}$ matrix
\begin{equation}\label{AtkinLehner}
W_q:=\left(\begin{matrix} q^{\alpha} a&b\\N_E c & q^{\alpha}d\end{matrix}\right),
\end{equation}
where $q^{\alpha} || N_E$ and $a,b,c,d\in\Z$. By Atkin-Lehner Theory, there is a $\lambda_q\in \{\pm 1\}$
for which $F_E |_2 W_q=\lambda_q F_E$.
The following result uses these involutions to give the
Fourier expansions of $\whZ_E(z)$ at cusps. When the level $N$ is squarefree, the next theorem gives the expansion at all cusps of $\G_0(N)$,
which can be explicitly computed using (\ref{addlaw}).

\begin{theorem}\label{thm2} If $q\vert N_E$, then
$$
\whZ_E(z) |_0 W_q=
\calZ_E^{+}(\lambda_q(\calE_E(z)-\Omega_q(F_E)))-\frac{\deg(\phi_E)}{4\pi ||F_E||^2}\cdot
\overline{\lambda_q(\calE_E(z)-\Omega_q(F_E))},
$$
where we have
$$
\Omega_q(F_E):=-2\pi i \int_{W_q^{-1}i \infty}^{i \infty} F_E(z) dz.
$$
\end{theorem}

\begin{remark} In particular, we have $\Omega_{N_E}(F_E)=L(F_E,1).$ By the modular parameterization, we have that
$\wp(\Lambda_E;\mathcal E_E(z))$ is a modular function on $\Gamma_0(N_E)$.  We then have for each $q| N_E$ that $\Omega_q(F_E)\in r\Lambda_E$, where $r$ is a rational number. This can be seen by considering the constant term of $\wp(\Lambda_E;\mathcal E_E(z))$ at cusps. The constant term of $\wp(\Lambda_E;\mathcal E_E(z))$ is $\wp(\Lambda_E;\Omega_q(F_E))$ (see Section \ref{proofs} for more details). More generally, if $N_E$ is square free,  then $\Omega_q(F_E)$ maps to a rational torsion point of $E$.
\end{remark}

As these facts illustrate, the harmonic Maass form $\whZ_E(z)$ and the mock
modular form $\whZ^{+}_E(z)$ encode the degree of the modular parameterization $\phi_E$,
which in turns gives information about the congruence number $r_E$, and it
encodes information about $\Q$-rational torsion.

By the work of Bruinier, Rhoades and the third author \cite{BORh}
and Candelori \cite{Candelori},
the coefficients of $\whZ_E^{+}(z)$ are $\Q$-rational when $E$ has complex multiplication.
For example, consider the elliptic curve $E\colon y^2 + y = x^3 - 38x + 90$ of conductor $361$ with CM in the field $K=\Q(\sqrt{-19})$.
We find
$$
F_E(z)=q - 2q^4 - q^5 + 3q^7 - 3q^9 - 5q^{11} + 4q^{16} - 7q^{17} + \dots
$$
and
$$
\zeta(\Lambda_E;\mathcal E_E(z))=q^{-1} + \frac{1}{2}q^{2} - \frac{7}{3}q^{3} + \frac{12}{5}q^{5} + 4q^{6} - \frac{6}{7}q^{7} - \frac{27}{4}q^{8} - \frac{13}{3}q^{9} + \frac{17}{2}q^{10} +\dots.
$$
As an illustration of this $\Q$-rationality, we find that
 $S(\Lambda_E)=-2$, which in turns gives
$$
\whZ^{+}_E(z)=q^{-1}+2q+\frac{1}{2}q^2-\frac{7}{3}q^3-q^4+2q^5+4q^6-\frac{27}{4}q^8-5q^9+\frac{17}{2}q^{10}+14q^{11}-\dots.
$$
This power series enjoys some deep $p$-adic properties with respect to Hecke operators. For example, it turns out that
$$\lim_{n\rightarrow +\infty} \dfrac{\left[ q\frac{d}{dq} \zeta (\Lambda_E;\mathcal E_E(z))\,\right]|T(5^n)}{a_E(5^n)}=-2F_E(z)
$$
as a $5$-adic limit.
To illustrate this phenomenon we offer:
$$\begin{array}{lclcl}
\dfrac{\left[ q\frac{d}{dq} \zeta (\Lambda_E;\mathcal E_E(z))\,\right]|T(5)}{a_E(5)}+2F_E(z)&=
 5q^{-5} - 20q - 85q^{2} - 430q^{3} - \dots&\equiv 0\pmod{5}\\ \ \ \\
\dfrac{\left[ q\frac{d}{dq} \zeta (\Lambda_E;\mathcal E_E(z))\,\right]|T(5^2)}{a_E(5^2)}+2F_E(z)&=
\frac{25}{4}q^{-25} - \frac{9525}{4}q - 2031975q^{2}  - \dots&\equiv 0\pmod{5^2}\\ \ \ \\
\dfrac{\left[ q\frac{d}{dq} \zeta (\Lambda_E;\mathcal E_E(z))\,\right]|T(5^3)}{a_E(5^3)}+2F_E(z)&=
-\frac{125}{9}q^{-125} - 89698470642375q  +\dots&\equiv 0\pmod{5^3}.
\end{array}$$

Our next result explains this phenomenon. There are such $p$-adic formulas for every $E$ provided that
$p\nmid N_E$ has the property that $p\nmid a_E(p)$ (i.e. $p$ is ordinary).
In analogy with recent work of Guerzhoy, Kent and the third author
\cite{GKO}, we obtain the following formulas.

\begin{theorem}\label{padicformula}
If $p\nmid N_E$ is ordinary, then there is a constant $\mathfrak{S}_E(p)$ for which
$$
\lim_{n\rightarrow +\infty} \frac{\left[ q\frac{d}{dq} \zeta (\Lambda_E;\mathcal E_E(z))\,\right]|T(p^{n})}{a_E(p^n)}=\mathfrak{S}_E(p) F_E(z).
$$
\end{theorem}

\begin{remark}
If $E$ has CM in Theorem~\ref{padicformula}, then $\mathfrak{S}_E(p)=S(\Lambda_E)$ as rational numbers. In other cases $S(\Lambda_E)$ is expected to be transcendental,
and one can interpret $\mathfrak{S}_E(p)$ as its $p$-adic expansion.
\end{remark}

The harmonic Maass forms $\whZ_E(z)$ also encode much information about Hasse-Weil $L$-functions. The seminal works by Birch and Swinnerton-Dyer \cite{BSD1, BSD2} give an indication of
this role in the case of CM elliptic curves. They obtained beautiful formulas for $L(E,1)$, for certain CM elliptic curves, as finite sums of numbers involving special values of $\zeta(\Lambda_E,s)$. Such formulas have been generalized by many authors
for CM elliptic curves (for example, see the famous papers by Damerell \cite{Damerell1, Damerell2}), and these generalizations have played a central role in the study of the arithmetic of CM elliptic curves.

Here we obtain results which show that the arithmetic of Weierstrass zeta-functions
gives rise to deep information which hold for all elliptic curves $E/\Q$, not just
those with CM.
We prove that the canonical harmonic Maass forms $\whZ_E(z)$
``encode'' the vanishing and nonvanishing of the central values
$L(E_D,1)$ and central derivatives $L'(E_D,1)$ for the quadratic twist elliptic curves $E_D$ of all modular elliptic curves.

The connection between these values and the theory of harmonic Maass forms was first made
by Bruinier and the third author \cite{BruinierOno}.
Their work proved that there
are weight 1/2 harmonic Maass forms whose coefficients give exact formulas
for $L(E_D,1)$, and which also encode the vanishing of $L'(E_D,1)$.
For central $L$-values their work relied on deep previous results of
Shimura and Waldspurger.
In the case of central derivatives, they made use of the theory of generalized
Borcherds products and the Gross-Zagier Theorem.
Bruinier \cite{BruinierPeriods} has
recently refined this work by obtaining exact formulas involving periods
of algebraic differentials.

The task of computing these weight 1/2 harmonic Maass forms has been nontrivial.
Natural difficulties arise (see \cite{BrStr}). These weight 1/2 forms are preimages under $\xi_{1/2}$ of certain weight 3/2 cusp forms, and
as mentioned earlier, there are infinitely many such preimages. Secondly, the methods implemented to date for constructing such
forms have relied on the theory of Poincar\'e series, forms whose coefficients are described as infinite sums of Kloosterman sums
weighted by Bessel functions. Establishing the convergence of these expressions can already pose difficulties. Moreover, there are infinitely many linear relations among Poincar\'e series. For a discussion of the relations among Maass-Poincar\'e series, the interested reader is referred to \cite{RhoadesPoincare}.

Here we circumvent these issues.
We construct canonical weight 1/2 harmonic Maass forms by making use of the canonical weight 0 harmonic Maass form $\whZ_E(z)$.
More precisely, we define a twisted theta lift using the usual Siegel theta function modified by a simple polynomial. This function was studied by H\"ovel \cite{Hoevel} in his Ph.D. thesis. The twisted lift $\calI_{\Delta,r}(\bullet;z)$ (see Section~\ref{ThetaLifts})
then maps weight 0 harmonic Maass forms to weight 1/2 harmonic Maass forms. Here $\Delta$ is a fundamental discriminant and $r$ is an integer satisfying $r^2\equiv\Delta\pmod{4N_E}$. For simplicity, we drop the dependence on $\Delta$ and $r$ in the introduction.
The canonical weight 1/2 harmonic Maass form we define is
\begin{equation}
f_E(z):=\calI\left(\whZ_E^*(z)-M^*_E(z);z\right),
\end{equation}
where $\whZ_E^*(z)$ and $M^*_E(z)$ denote a suitable normalization of $\whZ_E(z)$ and $M_E(z)$ (see Section \ref{GeneralTheorem}). The normalization originates from the fact that we need the rationality of the principal part of $f_E$ and we need to substract constant terms from the input. Following (\ref{fourier}), we let
\begin{equation}\label{fE}
f_E(z)=f_E^{+}(z)+f_E^{-}(z)=\sum_{n\gg -\infty}c_E^{+}(n)q^n+
\sum_{n<0} c_E^{-}(n) \G\left(\frac12,4\pi \abs{n}y\right)q^n.
\end{equation}

Although we treat the general case in this paper (see Theorem~\ref{GeneralCase}), to simplify exposition, in the remainder of the introduction we shall assume that $N_E=p$ is prime, and we shall assume
that the sign of the functional equation of $L(E,s)$ is $\epsilon(E)=-1$.
Therefore, we have that $L(E,1)=0$.
The coefficients of $f_E$ then satisfy the following theorem.

\begin{theorem}\label{thm3}
Suppose that $N_E=p$ is prime and that $\epsilon(E)=-1$. Then we have that
$f_E(z)$ is a weight 1/2 harmonic Maass form on $\Gamma_0(4p)$. Moreover, the following are true:

\noindent
(1) If $d<0$ is a fundamental discriminant for which $\leg{d}{p}=1$, then
$$L(E_d,1)=0\ \ {\text {if and only if}}\ \ c_E^{-}(d)=0.
$$

\noindent
(2) If $d>0$ is a fundamental discriminant for which $\leg{d}{p}=1$, then
$$L'(E_d,1)=0\ \ {\text {if and only if}}\ \  c_E^{+}(d)\ \
{\text {is in  }\ \Q}.
$$
%Moreover, $c_E^{+}(d)$ is algebraic if and only if $c_E^{+}(d)$ is in $\Q$.
\end{theorem}

\begin{remark}  Assume that $E$ is as in Theorem~\ref{thm3}. By work of
Kolyvagin
\cite{Kolyvagin} and Gross and Zagier \cite{GZ}
on the Birch and Swinnerton-Dyer Conjecture,
we then have the following for fundamental
discriminants $d$:
\begin{enumerate}
\item If $d<0$, $\leg{d}{p}=1$, and $c_E^{-}(d)\neq 0$, then the rank of $E_d(\Q)$ is 0.

\smallskip
\item If $d>0$, $\leg{d}{p}=1$, and $c_E^{+}(d)$ is transcendental, then the rank of $E_d(\Q)$ is 1.

\end{enumerate}
Criterion (1) is analogous to Tunnell's \cite{Tunnell} work on
the {\it Congruent Number Problem}.
\end{remark}

\begin{remark}
Theorem~\ref{thm3} follows from exact formulas. In particular, Theorem~\ref{thm3} (1)
follows from the exact formula
$$
L(E_d,1)=8\pi^2 ||F_E||^2\cdot ||g_E||^2\cdot \sqrt{\frac{|d|}{p}}\cdot
c_E^{-}(d)^2.
$$
Here $g_E$ is the weight 3/2 cusp form which is the image of $f_E(z)$
under the differential operator $\xi_{\frac{1}{2}}$ (see (\ref{defxi})). More precisely, we require that $\xi_{1/2}(f_E)=||g_E||^{-2}g_E$ (resp. $\xi_{1/2}(f_E)\in\R\cdot g_E$).
Theorem~\ref{thm3} (2) is also related to exact formulas, ones involving periods
of algebraic differentials.
Recent work by Bruinier \cite{BruinierPeriods} establishes that
$$
c_E^{+}(d)=\frac{\Re \int_{C_{F_E}}\zeta_{d}(f_E)}{\sqrt{d}\int_{C_{F_E}}\omega_{F_E}},
$$
where $\zeta_{d}(f_E)$ is the normalized differential of the third kind for a certain divisor associated to $f_E$ and $\omega_{F_E}=2\pi i F_E(z) dz$. Here $C_{F_E}$ is a generator of the $F_E$-isotypical component of the first homology of $X$.
The interested reader should consult \cite{BruinierPeriods} for further details.
\end{remark}

Theorem~\ref{thm3} follows from a general result on the theta lift $\calI(\bullet,z)$ we define in Section~\ref{ThetaLifts}.
Earlier work of Bruinier and Funke \cite{BrFu06}, the first author and Ehlen \cite{AE}, and more recent work of Bruinier and the first and third authors
\cite{Alfes, BruinierOno2}, consider similar theta lifts which implement
 the Kudla-Millson theta function as the kernel function. Those works give lifts which map weight $-2k$ forms
 to weight $3/2+k$ forms when $k$ is even. For odd $k$, these lifts map to weight $1/2-k$ forms. The new theta lift here makes use of the usual Siegel theta kernel which is modified with a simple polynomial. Using this weight $1/2$ function  H\"ovel \cite{Hoevel} defined  a theta lift going in the direction ``opposite'' to ours, i.e. from forms for the symplectic group to forms for the orthogonal group.
% This function, which has weight $1/2$,  was recently studied by H\"ovel \cite{Hoevel} in his Ph.D. thesis. He also defined a theta lift but going in the opposite direction, i.e. from forms for the symplectic group to forms for the orthogonal group.

We prove that the lift we consider
 maps weight $0$ forms to weight $1/2$ forms. Moreover, it satisfies
Hecke equivariant commutative diagrams, involving $\xi_0, \xi_{1/2}$ and  the Shintani lift, of the form:
\begin{gather*}
  \xymatrix@1{ \whZ_E^*(z)-M_E(z) \ar[d]^-{\calI} \ar[r]^-{\xi_0} &  F_E \ar[d]^-{\text{Shin}}
\\
\calI(\whZ_E^*(z)-M^*_E(z);\tau)\ar[r]^-{\xi_{1/2}} & \R\cdot g_E.}
\end{gather*}
Here $g_E$ is the weight $3/2$ cusp form in Remark 8.
This diagram explains our main motivation for finding ``canonical'' lifts under $\xi_0$, and given the work of Waldspurger \cite{Wa}, Kohnen-Zagier \cite{KZ}, and Bruinier and the third author \cite{BruinierOno} interpreting the arithmetic of the bottom row of this diagram, we see that finding such lifts has deep applications to the arithmetic (specifically, the first two Taylor coefficients of $L$-functions) of elliptic curves. 

\begin{remark}
It turns out that the coefficients $c_E^{+}(n)$ of $f_E(\tau)$ are ``twisted traces'' of the singular moduli for the weight 0
harmonic Maass form $\whZ_E^*(z)-M^*_E(z)$. This is Theorem~\ref{thm:trace}. This phenomenon is not new.
Seminal works by Zagier \cite{Za2} and Katok and Sarnak \cite{KS},
  followed by subsequent works by Bringmann, Bruinier,
Duke,  Funke, Imamo\={g}lu,  Jenkins, Miller,  Pixton, and  T\'oth \cite{BO3, BrFu06, BJO, Duke, DIT1, DIT2, DJ,  MillerPixton}, among many others, give situations where Fourier coefficients are
such traces. In particular, we obtain (vector valued versions of) the generating functions for the twisted traces of the $j$-invariant that Zagier called $f_d$, where $d$ is a fundamental discriminant, in \cite{Za2}. We explain this in more detail in Example \ref{Zagier}.
\end{remark}

\begin{example}  In Section~\ref{EXAMPLES} we shall consider the conductor 37 elliptic curve
$$
E: \ \  y^2-y=x^3-x.
$$
The sign of the functional equation of $L(E,s)$ is $-1$, and $E(\Q)$ has rank 1.
The table below illustrates Theorem~\ref{thm3}, and its implications for ranks of elliptic curves.
 \begin{table}[h]
\begin{tabular}{|r|ll|ll|lc|}
\hline %\rule[-3mm]{0mm}{8mm}{8mm}
$d$ && $c^+(d)$ && $L'(E_{d},1)$ && $\rk(E_{d}(\Q))$ \\
\hline% \rule[-3mm]{0mm}{8mm}
$1$  && $-0.2817617849\dots$ && $0.3059997738\dots$ && $1$ \\%26731$\\
$12$  && $-0.4885272382\dots$  && $4.2986147986\dots$&& $1$\\%69658$\\
$21$  && $-0.1727392572\dots$  && $9.0023868003\dots$&& 1\\%04156$\\
$28$ && $-0.6781939953\dots$ && $ 4.3272602496\dots$ && 1\\
$33$ &&$\ \ \, 0.5663023201\dots$ && $3.6219567911\dots$&& 1 \\%73346$\\
%$-40$ && $-0.577676447138672293721\dots$ \\%10457$\\
$\ \ \ \ \ \vdots$ && $\ \ \ \ \ \vdots$&&$\ \ \ \ \ \vdots$ && \vdots\\
$1489$&& $\ \ \ \ \, 9$  && $\ \ \ \ \ 0$&& 3\\%50497$\\
$\ \ \ \ \  \vdots$&&$\ \ \ \ \  \vdots$&&$\ \ \ \ \ \vdots$&& \vdots  \\
$4393$&& $\ \ \ \,  66$ && $\ \ \ \ \ 0$&& 3 \\ %1.8634999382929359284261114E-30$\\
\hline
\end{tabular}
\end{table}
For the $d$ in the table we have that the sign of the functional equation of $L(E_{d},s)$ is $-1$. Therefore, if $L'(E_{d},1)\neq 0$, then we have that $\ord_{s=1}(L(E_{d},s))=1$, which then
implies that $\rk(E_{d}(\Q))=1$ by Kolyvagin's Theorem. For such $d$, Theorem~\ref{thm3} asserts that $L'(E_{d},1)=0$ if and only if
$c_E^{+}(d)\in \Q$. Therefore, for these $d$ the Birch and Swinnerton-Dyer Conjecture implies that $\rk(E_{d}(\Q))\geq 3$ is odd if and only if
 $c_E^{+}(d)\in \Q$. We note that for $d\in \{1489, 4393\}$, we find\footnote{These computations were done using {\tt Sage}\cite{sage} by Bruinier and Str\"omberg in \cite{BrStr}. Stephan Ehlen obtained the same numbers using our results (also using {\tt Sage}).} that the curves have rank 3.

\end{example}

The paper is organized as follows.
In Section~\ref{WeierstrassStuff} we prove Theorem~\ref{thm1}, \ref{thm2}, and \ref{padicformula}.
In Section~\ref{VectorValued} we recall basic facts about the Weil representation and
vector-valued harmonic Maass forms and introduce the relevant theta functions.
This is required because we shall state Theorem~\ref{GeneralCase}, the general version
of Theorem~\ref{thm3},
in terms of vector-valued harmonic
Maass forms. In Section~\ref{sec:lifts} we construct the theta lift $\calI(\bullet;\tau)$. In Section~\ref{GeneralTheorem} we
state and prove the general form of Theorem~\ref{thm3}.
In Section~\ref{EXAMPLES} we give a number of examples which illustrate the theorems proved in this paper.

\section*{Acknowledgements}
\noindent The authors thank Jan Bruinier and Pavel Guerzhoy for
helpful discussions. We also thank Stephan Ehlen for his numerical calculations in this paper, his corrections and many fruitful conversations. Furthermore, we are grateful to the referee for their many useful suggestions which improved the exposition of this paper.

\section{Weierstrass Theory and the proof of Theorems~\ref{thm1}, \ref{thm2} and \ref{padicformula}}
\label{WeierstrassStuff}

Here we recall the essential features of the Weierstrass theory of elliptic curves.
After recalling these facts, we then prove Theorems~\ref{thm1} and \ref{thm2}.

\subsection{Basic facts about Weierstrass theory}
%{\bf Larry and Michael: Put it in...}
As noted in the introduction, the analytic parameterization $\C/\Lambda_E\cong E$ of an elliptic curve is given by $\z\to P_{\z}=(\wp(\Lambda_E;\z), \wp'(\Lambda_E;\z))$. By evaluating the Weierstrass $\wp$-function at the Eichler integral given in (\ref{EichlerIntegral}), this analytic parameterization becomes the modular parameterization. The Eichler integral is not modular, however its obstruction to modularity is easily characterized. The map $\Psi_E\colon\Gamma_0(N)\to \C$ given by
\begin{equation}\label{Psi}
\Psi_E(\gamma):=\mathcal E_E(z)-\mathcal E_E(\gamma z)
\end{equation}
is a homomorphism of groups. Its image in $\C$ turns out to be the lattice $\Lambda_E$. Hence, since $\wp(\Lambda_E;\z)$ is invariant on the lattice, the map $\wp(\Lambda_E;\mathcal E_E(z))$
parameterizes $E$ and is also a modular function.

Theorems ~\ref{thm1} and \ref{thm2} rely on a similar observation, but in this case involving the Weierstrass $\zeta$-function. Unlike the Weierstrass $\wp$-function, the $\zeta$-function itself is not lattice-invariant. However, Eisenstein \cite{Weil} observed that it could be modified to become lattice-invariant but this modification necessarily sacrifices holomorphicity.

\subsection{Proofs of Theorems~\ref{thm1} and \ref{thm2}}\label{proofs}

We now prove Theorems~\ref{thm1} and  \ref{thm2}.

\begin{proof}[Proof of Theorem~\ref{thm1}]
Eisenstein's  modification to the $\zeta$-function is given by
\begin{equation}\label{ZetaStar}
\zeta(\Lambda_E;\z)-S(\Lambda_E)\z-\frac{\pi}{a(\Lambda_E)}\overline{\z}.
\end{equation}
Here $S$ is as in (\ref{S}) and $a(\Lambda_E)$ is the area of a fundamental parallelogram for $\Lambda_E.$

Using the formula
\begin{equation}\label{area}
a(\Lambda_E)=\frac{4\pi^2||F_E||^2}{\mathrm{deg}(\phi_E)},
\end{equation}
 we have that the function $\mathfrak Z_E(\z)$ defined in (\ref{CorrectedZeta}) above is Eisenstein's corrected $\zeta$-function and is lattice-invariant. Formula (\ref{area}) was first given by Zagier \cite{ZagierModDegree} for prime conductor and generalized by Cremona for general level \cite{Cremona}. Since $\mathfrak Z_E(\z)$ is lattice-invariant,
 $\whZ_E(z)$, defined by (\ref{HMF}), is modular.

Part $(1)$ of Theorem \ref{thm1} follows by noting that $\mathfrak Z_E(\z)$ diverges precisely for $\z\in \Lambda_E.$ This divergence must result from a pole in the holomorphic part, $\mathfrak Z_E^+(\z).$

In order to establish part $(2)$, we consider the modular function $\wp(\Lambda_E;\mathcal E_E(z))$. We observe that $\wp(\Lambda_E;\mathcal E_E(z))$ is meromorphic with poles precisely for those $z$ such that $\mathcal E_E(z)\in \Lambda_E.$ We claim $\wp(\Lambda_E;\mathcal E_E(z))$ may be decomposed into modular functions with algebraic coefficients, each with only a simple
 pole at one such $z$ and possibly at cusps. This follows from a careful inspection of the standard proof that $M_0^!(N)=\C\left(j(z),j(Nz)\right)$. For example, following the proof of Theorem 11.9 in \cite{Cox}, one obtains an expression for the given modular function in terms of a function $G(z)$ and a modular function with rational coefficients. The function $G(z)$ clearly lies in $\overline{\Q}\left(j(z),j(Nz)\right)$ whenever we start with a modular function with algebraic coefficients at all cusps, from which the claim follows easily.
 
 These simple modular functions may then be combined appropriately to construct the function $M_E(z)$ to cancel the poles of $\widehat{\mathfrak {Z}}_E^+(z)$, and the remainder of the proof of (3) then
follows from straightforward calculations. In particular, it is enough to plug in the transformation equation for $\mathcal E_E$ into Weierstrass' completion given in \eqref{ZetaStar} (see \eqref{Psi} and the following comments stating that the image of this map is the period lattice of $E$) and compute the image under $\xi_0$ of the non-holomorphic part, as given in \eqref{CorrectedZeta}. 
\end{proof}

Using the theory of Atkin-Lehner involutions (in particular, the reader may find the relevant facts in \cite{AtkinLehnerPaper}), we now prove Theorem~\ref{thm2}.

\begin{proof}[Proof of Theorem~\ref{thm2}]
%{\bf Larry and Michael: Put it in...}
Recall that by classical theory of Atkin-Lehner, every newform of level $N_E$ is an eigenform of the Atkin-Lehner involution
$$
W_q=\begin{pmatrix} q^\alpha a&b\\Nc&q^\alpha d\end{pmatrix},
$$
for every prime power $q|| N_E$, with eigenvalue $\pm1$.
We note that
$$\whZ_E(z)|_0W_q=\mathfrak Z_E(\Lambda_E;\mathcal E_E(z)|_0W_q).$$
 It suffices to show $\mathcal E_E(z)-\lambda_q\mathcal E_E(z)|W_q$ is equal to $\Omega_q(F_E).$ To this end note that

\begin{eqnarray}\label{EichlerInv}
\ \ \mathcal E_E(z)-\lambda_q\mathcal E_E(z)|W_q &=& -2\pi i\left[  \int_{z}^{i\infty}F_E(z)dz-\lambda_q\int_{W_q z}^{i\infty}F_E(z)dz\right]\\
\nonumber &=& -2\pi i\left[ \int_{z}^{i\infty}F_E(z)dz-{\lambda_q} \int_{z}^{W_q^{-1}i\infty}\frac{\det(W_q)}{(Ncz+q^\alpha d)^2}F_E(W_qz)dz\right]\\
\nonumber&=& -2\pi i\left[  \int_{z}^{i\infty}F_E(z)dz+\lambda_q^2\int^{z}_{W_q^{-1}i\infty}F_E(z)dz\right]\\
\nonumber &=& -2\pi i\int^{i\infty}_{W_q^{-1}i\infty}F_E(z)dz.
\end{eqnarray}
We note that if $\Omega_q(F_E)$ is in the lattice, then we may ignore this term, and we see that $\whZ_E(z)$ is an eigenfunction for the involution $W_q$. Otherwise, $\whZ_E(z)|_0W_q$ has a constant term equal to $\mathfrak Z_E(\Omega_q(F_E)).$
\end{proof}

\subsection{Proof of Theorem~\ref{padicformula}}

The proof of Theorem \ref{padicformula} is similar to recent work of Guerzhoy, Kent and the third author \cite{GKO}.
We will need the  following  proposition.
\begin{proposition}\label{vanishingMFs}
Suppose that $R(z)$ is a meromorphic modular function on $\Gamma_0(N)$  with $\Q$-rational coefficients. If $p\nmid N$ is prime, then there is an A such that
$$\ord_p\left(q\frac{d}{dq} R| T(p^n)\right )\geq n- A.$$
\end{proposition}
\begin{proof}
For convenience, we let $R(z)=\sum_{n\gg -\infty}a(n)q^n$.
We first show that the coefficients $a(n)$ of $R$ have bounded denominators. In other words, we have
that $A := \inf_n (\ord_p(a(n))) < \infty$. Indeed, we can always multiply $R$ with an appropriate
power of $(z)$ and a monic polynomial in $j(z)$ with rational coefficients to obtain a cusp form of positive integer weight and rational coefficients.  The resulting Fourier coefficients will have
bounded denominators by Theorem~3.52 of \cite{Shimura}.  One easily checks that dividing by the power of $\Delta(z)$ and this polynomial in $j(z)$ preserves the boundedness. The proposition now follows easily from
$$
\left(q\frac{d}{dq} R\right)| T(p^n)=\sum_{m\gg \infty}\sum_{j=0}^{\min\{\ord_p(m),n\}} p^{n-j}ma(p^{n-2j}m)q^m.
$$
\end{proof}
\begin{remark}
Proposition~\ref{vanishingMFs} is analogous to Proposition 2.1 of \cite{GKO} which concerns Atkin's $U(p)$ operator.
\end{remark}

\begin{proof}[Proof of Theorem~\ref{padicformula}] We first consider the case where $E$ has CM.
 Suppose $D<0$ is the discriminant of the imaginary quadratic field $K$. The nonzero coefficients of $F_E(z)$ are supported on powers $q^n$ with $\chi_D(n):=\left(\frac{D}{n}\right)\neq -1$. Let  $\varphi_D$ be the trivial character modulo $|D|$. We construct the modular  function
\begin{equation}\label{Z}
\mathcal Z_E(z) = \frac{1}{2}\left(\whZ_E|\varphi_D+\whZ_E|\chi_D\right).
\end{equation}
 Since the coefficients of the nonholomorphic part of $\whZ_E(z)$ are supported on powers $q^{-n}$ with $\chi_D(-n)\neq 1$, we see that the twisting action in the definition of $\mathcal Z_E(z)$ kills the nonholomorphic part. Therefore, $\mathcal Z_E(z)$ is a meromorphic modular function on
 $\Gamma_0(ND^2)$ whose nonzero coefficients are supported on $q^m$ where $\chi_D(m)=1$, and are equal to the original coefficients of
 $\whZ_E^{+}(z)$.

We now aim to prove the following $p$-adic limits:
\begin{equation}\label{vanishingMaass} \lim_{n\rightarrow +\infty} {\left [q\frac{d}{dq}( \whZ_E (z))\right]|T(p^{n})}=
\lim_{n\rightarrow +\infty} {\left [q\frac{d}{dq}( \whZ_E (z) - \mathcal Z_E(z))\right]|T(p^{n})}=0.
\end{equation}
By Proposition~\ref{vanishingMFs}, the two limits are equal, and so it suffices to prove the vanishing of the second limit.

Since $\chi_D(p^n)=1$, it  follows that the coefficients of $q^{p^n}$ (including $q^1$) in $\whZ^{+}_E(z)-\mathcal{Z}_E(z)$ all vanish.
Therefore the coefficient of $q^1$ for each $n$ in  the second limit of (\ref{vanishingMaass}) is zero.
Since the principal part of $\whZ_E(z)-\mathcal{Z}_E(z)$ is  $q^{-1}$, the
principal parts in the second limit $p$-adically tend to $0$ thanks to the definition of the Hecke operators $T(p^n)$.

Suppose that $m>1$ is coprime to $N_E$. Then note that $F_E$ is an eigenfunction for the Hecke operator $T(m)$ with eigenvalue $a_E(m).$ Since the nonholomorphic part of $\whZ_E(z)$ is the period integral of $F_E(z)$, it follows that $Q_m(z):=m\whZ_E(z)| T(m)-a_E(m)\whZ_E(z)=m\whZ_E^+(z)| T(m)-a_E(m)\whZ_E^+(z)$ is a meromorphic modular function. Note that the functions $q\frac{d}{dq} Q_m(z)$ have denominators that are bounded independently of $m$.
This follows from the proof of Proposition~\ref{vanishingMFs} and the fact that (see Theorem 1.1 of \cite{BORh})
$q\frac{d}{dq} \whZ_E(z)$ is a weight 2 meromorphic modular form. Since Hecke operators commute, we have
\begin{displaymath}
\begin{split}
{\left[q\frac{d}{dq}  \whZ_E^+ (z)\right]|T(p^{n})T(m)}
 ={\left[ q\frac{d}{dq} ( a_E(m)\whZ_E^+ (z) +Q_m(z))\right]|T(p^{n})}.
\end{split}
\end{displaymath}
Modulo any fixed power of $p$, say $p^t$, Proposition~\ref{vanishingMFs} then implies that
\begin{displaymath}
\begin{split}
{\left[q\frac{d}{dq}  \whZ_E^+ (z) \right]|T(p^{n})T(m)}
 \equiv {a_E(m)\cdot \left[ q\frac{d}{dq} \whZ_E^+ (z)\right]|T(p^{n})} \pmod{p^t},
\end{split}
\end{displaymath}
for sufficiently large $n$.
In other words, we have that
$\left [q\frac{d}{dq} \whZ_E^+ (z) \right] | T(p^{n})$
 is congruent to a Hecke eigenform for $T(m)$ modulo $p^t$ for sufficiently large $n$.  By Proposition~\ref{vanishingMFs} again,
 we have that $\left [q\frac{d}{dq} (\whZ_E^+ (z)-\mathcal{Z}_E(z)) \right] | T(p^{n})$ is an eigenform of $T(m)$ modulo $p^t$ for sufficiently large
 $n$.
Obviously, this conclusion holds uniformly in $n$ for all $T(m)$ with $\gcd(m,N_E)=1$.

Generalizing this argument in the obvious way to incorporate Atkin's $U$-operators (as in \cite{GKO}), we conclude that these forms
are eigenforms of all the Hecke operators.
By the discussion above, combined with the fact that the constant terms vanish after applying $q\frac{d}{dq}$, these eigenforms are congruent to $0+ O(q^2)\pmod{p^t}$.
Such an eigenform must be identically $0\pmod{p^t}$, thereby establishing (\ref{vanishingMaass}).

To complete the proof in this case, we observe that $p\nmid a_E(p^n)$ for any $n$. This follows from the recurrence relation on $a_E(p^n)$ in $n$, combined with
the fact that $p\nmid a_E(p)$ since $p$ is split in $K$.
By (\ref{vanishingMaass}) we have that
\begin{equation}\label{limit}
\lim_{n\rightarrow +\infty} \frac{\left[ q\frac{d}{dq}  (\whZ_E^+ (z) )\right]|T(p^{n})}{a_E(p^n)}=0.
\end{equation}
The proof now follows from the identities
\begin{displaymath}
\whZ_E^+ (z)=\zeta(\Lambda_E;\mathcal E_E(z))-S(\Lambda_E)\mathcal E_E(z) \ \ \ {\text {\rm and}}\ \ \
F_E(z)= q\frac{d}{dq} \mathcal{E}_E(z).
\end{displaymath}

The proof for $E$ without CM is nearly identical. We replace $\whZ_E^{+}(z)$ by $\whZ_E^{+}(z)+S(\Lambda_E)\mathcal{E}_E(z)$, which has
$\Q$-rational coefficients. In (\ref{limit}) the limiting value of $0$ is replaced by a constant multiple of $F_E(z)$.
\end{proof}

\section{Vector valued harmonic Maass forms}
\label{VectorValued}

To ease exposition, the results in the introduction were stated
using the classical language of half-integral weight modular forms.
 To treat the case of general levels and functional equations, it will
be more convenient to work with vector-valued forms and certain Weil
representations. Here we recall this framework, and we discuss
important theta functions which will be required in the section to
define the theta lift $\calI(\bullet;\tau)$.
In particular, the reader will notice in
Section~\ref{sec:weil} that harmonic Maass forms are defined with respect
to the variable $\tau\in \H$ instead of the variable $z$ as in Section~\ref{sect:intro}.
Moreover, we shall let $q:=e^{2\pi i \tau}$. The modular parameter will
always be clear in context. The need for multiple modular variables arises from the structure of the theta lift. As a rule of thumb, $\tau$ shall be the modular
variable for all the half-integral weight forms in the remainder of this paper.

For a positive integer $N$ we consider the rational quadratic space of signature $(1,2)$ given by
\[
V:=\left\{\lambda=\begin{pmatrix} \lambda_1 &\lambda_2\\ \lambda_3& -\lambda_1\end{pmatrix}; \lambda_1,\lambda_2,\lambda_3 \in \Q\right\}
\]
and the quadratic form $Q(\lambda):=N\text{det}(\lambda)$.
The associated bilinear form is $(\lambda,\mu)=-N\text{tr}(\lambda\mu)$ for $\lambda,\mu \in V$.

We let $G=\mathrm{Spin}(V) \simeq \Sl_2$, viewed as an algebraic group over $\Q$
and write $\overline{\G}$ for its image in $\mathrm{SO}(V)\simeq\mathrm{PSL}_2$.
By $D$ we denote the associated symmetric space. It can be realized as the Grassmannian of lines
in $V(\R)$ on which the quadratic form $Q$ is positive definite,
\[
D \simeq \left\{z\subset V(\R);\ \dim z=1 \text{ and } Q\vert_{z} >0 \right\}.
\]
Then the group $\Sl_2(\Q)$ acts on $V$ by conjugation
\[
  g\textbf{.}\lambda :=g \lambda g^{-1},
\]
for $\lambda \in V$ and $g\in\Sl_2(\Q)$. In particular, $G(\Q)\simeq\Sl_2(\Q)$.

We identify the symmetric space $D$ with the upper-half of the complex plane $\h$ in the usual way,
and obtain an isomorphism
between $\h$ and $D$ by
\[
 z \mapsto \R \lambda(z),
 \]
where, for $z=x+iy$, we pick as a generator for the associated positive line
 \[
 \lambda(z):=\frac{1}{\sqrt{N}y} \begin{pmatrix} -(z+\bar{z})/2 &z\bar{z} \\  -1 & (z+\bar{z})/2 \end{pmatrix}.
 \]
The group $G$ acts on $\h$ by linear fractional transformations and
the isomorphism above is $G$-equivariant.
Note that $Q\left(\lambda(z)\right)=1$ and $g\textbf{.}\lambda(z)=\lambda(gz)$ for $g\in G(\R)$.
Let $(\lambda,\lambda)_z=(\lambda,\lambda(z))^2-(\lambda,\lambda)$. This is the minimal majorant of $(\cdot,\cdot)$ associated with $z\in D$.

We can view $\G_0(N)$ as a discrete subgroup of $\mathrm{Spin}(V)$ and we write $M=\G_0(N) \setminus D$ for the attached locally symmetric space.

We identify the set of isotropic lines $\mathrm{Iso}(V)$ in $V(\Q)$
with $P^1(\Q)=\Q \cup \left\{ \infty\right\}$ via
\[
\psi: P^1(\Q) \rightarrow \mathrm{Iso}(V), \quad \psi((\alpha:\beta))
 = \mathrm{span}\left(\begin{pmatrix} \alpha\beta &\alpha^2 \\  -\beta^2 & -\alpha\beta \end{pmatrix}\right).
\]
The map $\psi$ is a bijection and $\psi(g(\alpha:\beta))=g.\psi((\alpha:\beta))$.
Thus, the cusps of $M$ (i.e. the $\G_0(N)$-classes of $P^1(\Q)$) can be identified with the $\G_0(N)$-classes of $\mathrm{Iso}(V)$.

If we set $\ell_\infty := \psi(\infty)$, then $\ell_\infty$ is spanned by$\lambda_\infty=\left(\begin{smallmatrix}0 & 1 \\ 0 & 0\end{smallmatrix}\right)$.
For $\ell \in \mathrm{Iso}(V)$ we pick $\sigma_{\ell} \in\SL_2(\Z)$
such that $\sigma_{\ell}\ell_\infty=\ell$.
%By $\alpha_\ell$ we denote the width of the cusp $\ell$.

Heegner points are given as follows.
For $\lambda\in V(\Q)$ with $Q(\lambda)>0$ we let
\[
D_{\lambda}= \mathrm{span}(\lambda) \in D.
\]
For $Q(\lambda) \leq 0$ we set $D_{\lambda}=\emptyset$.
We denote the image of $D_{\lambda}$ in $M$ by $Z(\lambda)$.

\subsection{A lattice related to $\mathbf{\G_0(N)}$}\label{sec:lattice}

We consider the lattice
\[
 L:=\left\{ \begin{pmatrix} b& -a/N \\ c&-b \end{pmatrix}; \quad a,b,c\in\Z \right\}.
\]
The dual lattice corresponding to the bilinear form $(\cdot,\cdot)$ is given by
\[
 L':=\left\{ \begin{pmatrix} b/2N& -a/N \\ c&-b/2N \end{pmatrix}; \quad a,b,c\in\Z \right\}.
\]
We identify the discriminant group $L'/L=:\dg$
with $\Z/2N\Z$, together with the $\Q/\Z$ valued quadratic form ${x \mapsto -x^2/4N}$.
The level of $L$ is $4N$.

For a fundamental discriminant $\Delta\in\Z$ we will consider the rescaled lattice $\Delta L$ together with the quadratic form $Q_\Delta(\lambda):=\frac{Q(\lambda)}{\abs{\Delta}}$. The corresponding bilinear form is then given by $(\cdot,\cdot)_\Delta = \frac{1}{\abs{\Delta}} (\cdot,\cdot)$. The dual lattice of $\Delta L$ with respect to $(\cdot,\cdot)_\Delta$ is equal to $L'$. We denote the discriminant group $L'/\Delta L$ by $\dgdelta$.

For $m \in \Q$ and $h \in \dg$, we let
\begin{equation*}
 L_{m,h}  = \left\{ \lambda \in L+h; Q(\lambda)=m  \right\}.
\end{equation*}
By reduction theory, if $m \neq 0$ the group $\G_0(N)$ acts on $ L_{m,h}$ with finitely many orbits.

We will also consider the one-dimensional lattice $K=\Z\left(\begin{smallmatrix} 1&0\\0&-1\end{smallmatrix}\right)\subset L$. We have $L=K+\Z\ell+\Z\ell'$ where $\ell$ and $\ell'$ are the primitive isotropic vectors
\[
 \ell=\begin{pmatrix} 0&1/N\\0&0\end{pmatrix},\quad\quad\quad\quad  \ell'=\begin{pmatrix} 0&0\\-1&0\end{pmatrix}.
\]
Then $K'/K\simeq L'/L$.

%%%%%%%%%%%%%%%%%%%%%%%%%%%%%%%%%%%%%%%%%%%%%%%%%%%%%%%%%%%%%%

\subsection{The Weil representation and vector-valued automorphic forms}\label{sec:weil}

By $\Mp_2(\Z)$ we denote the integral metaplectic group. It consists of pairs
$(\gamma, \phi)$, where $\gamma = {\smallabcd \in \Sl_2(\Z)}$ and $\phi:\h\rightarrow \C$
is a holomorphic function with $\phi^2(\tau)=c\tau+d$.
The group $\widetilde{\G}=\Mp_2(\Z)$ is generated by $S=(\smallSmatrix,\sqrt{\tau})$ and $T=(\smallTmatrix, 1)$. We let $\widetilde{\G}_\infty:=\langle T\rangle \subset \widetilde{\G}$.
We consider the Weil representation $\rho_\Delta$ of $\Mp_2(\Z)$
corresponding to the discriminant group $\dgdelta$ on the group ring $\C[\dgdelta]$,
equipped with the standard scalar product $\langle \cdot , \cdot \rangle$, conjugate-linear
in the second variable. We simply write $\rho$ for $\rho_1$.

Let $e(a):=e^{2\pi i a}$. We write $\e_\delta$ for the standard basis element of $\C[\dgdelta]$ corresponding to $\delta \in \dgdelta$.
The action of $\rho_\Delta$ on basis vectors of $\C[\dgdelta]$
is given by the following formulas for
the generators $S$ and $T$ of $\Mp_2(\Z)$
\begin{equation*}
 \rho_\Delta(T) \e_\delta = e(Q_\Delta(\delta)) \e_\delta,
\end{equation*}
and
\begin{equation*}
 \rho_\Delta(S) \e_\delta = \frac{\sqrt{i}}{\sqrt{\abs{\dgdelta}}}
				\sum_{\delta' \in \dgdelta} e(-(\delta',\delta)_\Delta) \e_{\delta'}.
\end{equation*}
Let $k \in \frac{1}{2}\Z$, and let $A_{k,\rho_\Delta}$ be the vector space of functions
$f: \h \rightarrow \C[\dgdelta]$, such that for $(\gamma,\phi) \in \Mp_2(\Z)$ we have
\begin{equation*}
	f(\gamma \tau) = \phi(\tau)^{2k} \rho_\Delta(\gamma, \phi) f(\tau).
\end{equation*}
A smooth function $f\in A_{k,\rho_\Delta}$
is called a \textit{harmonic (weak) Maass form of weight $k$ with respect to the representation $\rho_\Delta$}
if it satisfies in addition (see \cite[Section 3]{BF}):
\begin{enumerate}
 \item $\Delta_k f=0$,
 \item the singularity at $\infty$ is locally given by the pole of a meromorphic function.
 \end{enumerate}
Here we write $\tau=u+iv$ with $u,v \in \R$, and
\begin{equation}\label{deflap}
\Delta_k=-v^2\left(\frac{\partial^2}{\partial u^2}+\frac{\partial^2}{\partial v^2}\right)
+ikv\left(\frac{\partial}{\partial u}+i\frac{\partial}{\partial v}\right)
\end{equation}
is the weight $k$ Laplace operator.
We denote the space of such functions by $H_{k,\rho_\Delta}$.

By $M^{\text{!}}_{k,\rho_\Delta} \subset H_{k,\rho_\Delta}$
we denote the subspace of weakly holomorphic modular forms. Recall that weakly holomorphic modular forms are meromorphic modular
forms whose poles (if any) are supported at cusps.

Similarly, we can define scalar-valued analogs of these spaces of automorphic forms.
In this case, we require analogous conditions at all cusps of $\G_0(N)$ in $(ii)$. We denote these spaces by
$H^{+}_{k}(N)$ and $M_k^{\text{!}}(N)$.

Note that the Fourier expansion of any harmonic Maass form uniquely decomposes into a holomorphic and a nonholomorphic part \cite[Section 3]{BF}
\begin{align*}
& f^{+}(\tau)=\sum\limits_{h\in
L'/L}\sum\limits_{\substack{n\in\Q\\n\gg -\infty}}c^{+}(n,h)q^n \mathfrak{e}_h
\\
& f^{-}(\tau)=\sum\limits_{h\in L'/L} \sum\limits_{n\in\Q }c^{-}(n,h)\G(1-k,4\pi \abs{n}v)q^n\mathfrak{e}_h,
\end{align*}
where $\G(a,x)$ denotes the incomplete $\G$-function. The first summand is called the holomorphic part of $f$, the second one the nonholomorphic part.

We define a differential operator $\xi_k$ by
\begin{equation}\label{defxi}
 \xi_k(f):=-2i v^k\overline{\frac{\partial}{\partial \bar{\tau}}f}.
\end{equation}
We then have the following exact sequence \cite[Corollary 3.8]{BF}
%\begin{gather}
  %\xymatrix@1{ 0 \ar[r] & M_{k,\rho_\Delta}^! \ar[r] & H_{k,\rho_\Delta}
%\ar[r]^-{\xi_{k}} & S_{2-k,\bar\rho_\Delta} \ar[r] & 0}.
%\end{gather}
\[0\longrightarrow M^!_{k,\rho_\Delta}\longrightarrow H_{k,\rho_\Delta}\buildrel{\xi_k}\over{\longrightarrow} S_{2-k,\bar\rho_\Delta}\longrightarrow0.\]

%%%%%%%%%%%%%%%%%%%%%%%%%%%%%%%%%%%%%%%%%%%%%%%%%%%%%%%%%%%%%%%%%%%%%%%%%%%%%%%%%%%%%%

\subsection{Poincar\'{e} series and Whittaker functions}\label{sec:pc}
We recall some facts on Poincar\'{e} series with exponential growth at the cusps following Section 2.6 of \cite{BruinierOno2}.

We let $k\in\frac12\Z$, and $M_{\nu,\mu}(z)$ and $W_{\nu,\mu}(z)$ denote the usual Whittaker functions (see p. 190 of \cite{Pocket}). For $s\in\C$ and $y\in\R_{>0}$ we put
\[
\mathcal{M}_{s,k}(y)=y^{-k/2}M_{-\frac{k}{2},s-\frac12}(y).
\]
We let $\G_\infty$ be the subgroup of $\G_0(N)$ generated by $\left(\begin{smallmatrix}1&1\\0&1\end{smallmatrix}\right)$. For $k\in\Z$, $m\in\N$, $z=x+iy\in\h$ and $s\in \C$ with $\Re(s)>1$, we define
\begin{equation}\label{def:poincare}
 F_m(z,s,k)=\frac{1}{2\G(2s)}\sum\limits_{\gamma\in\G_\infty\setminus\G_0(N)}\left[\M_{s,k}(4\pi my)e(-mx)\right]|_k\ \gamma.
\end{equation}
This Poincar\'{e} series converges for $\Re(s)>1$, and it is an eigenfunction of $\Delta_k$ with eigenvalue $s(1-s)+(k^2-2k)/4$. Its specialization at $s_0=1-k/2$ is a harmonic Maass  form \cite[Proposition 1.10]{Br}. The principal part at the cusp $\infty$ is given by $q^{-m}+C$ for some constant $C\in\C$. The principal parts at the other cusps are constant.

We now define $\C[L'/L]$-valued analogs of these series. Let $h\in L'/L$ and $m\in\Z-Q(h)$ be positive. For $k\in\left(\Z-\frac12\right)_{< 1}$ we let
\[
\calF_{m,h}(\tau,s,k)=\frac{1}{2\G(2s)}\sum\limits_{\gamma\in \widetilde{\G}_\infty\setminus\widetilde{\G}}\left. \left[\M_{s,k}(4\pi m y)e(-mx)\mathfrak{e}_h\right]\right|_{k,\rho}\ \gamma.
\]
%and for $k\in\left(\Z-\frac12\right)_{\geq 0}$ we let
%\[
%\F_{m,h}(\tau,s,k)=\frac{1}{2}\sum\limits_{\gamma\in \widetilde{\G}_\infty\setminus\widetilde{\G}}\left. \left[\M_{s,k}(4\pi m y)e(-mx)\mathfrak{e}_h\right]\right|_{k,\rho}\ \gamma.
%\]
The series $\calF_{m,h}(\tau,s,k)$ converges for $\Re(s)>1$ and it defines a harmonic Maass
  form of weight $k$ for $\widetilde{\G}$ with representation $\rho$. The special value at $s=1-k/2$
%, if $k\in\left(\Z-\frac12\right)_{< 1}$, resp. $s=k/2$, if $k\in\left(\Z-\frac12\right)_{\geq 1}$,
 is harmonic \cite[Proposition 1.10]{Br}. For $k\in\Z-\frac12$ the principal part is given by $q^{-m}\mathfrak{e}_h+q^{-m}\mathfrak{e}_{-h}+C$ for some constant $C\in\C[L'/L]$.

\begin{remark}
 If we let (in the same setting as above)
\[
\calF_{m,h}(\tau,s,k)=\frac{1}{2\G(2s)}\sum\limits_{\gamma\in \widetilde{\G}_\infty\setminus\widetilde{\G}}\left. \left[\M_{s,k}(4\pi m y)e(-mx)\mathfrak{e}_h\right]\right|_{k,\bar\rho}\ \gamma,
\]
then this has the same convergence properties. But for the special value at $s=1-k/2$, the principal part is given by  $q^{-m}\mathfrak{e}_h-q^{-m}\mathfrak{e}_{-h}+C$ for some constant $C\in\C[L'/L]$.
\end{remark}

%\begin{remark}
% For $k=0$ these Poincar\'{e} series series and the constant function span the space $H^+_{0,\rho}$ \cite[Proposition 1.12]{Br}.
%\end{remark}

\subsection{Twisted theta series}\label{sec:twistvec}
We define a generalized genus character for
$\delta = \left(\begin{smallmatrix} b/2N& -a/N \\ c&-b/2N \end{smallmatrix}\right) \in L'$.
From now on let $\Delta\in\Z$ be a fundamental discriminant and $r\in\Z$ such that $\Delta \equiv r^2 \ (\text{mod } 4N)$.

Then
\begin{equation*}
\chi_{\Delta}(\delta)=\chi_{\Delta}(\left[a,b,Nc\right]):=
\begin{cases}
\Deltaover{n}, & \text{if } \Delta | b^2-4Nac \text{ and } (b^2-4Nac)/\Delta \text{ is a}
\\
& \text{square mod } 4N \text{ and } \gcd(a,b,c,\Delta)=1,
\\
0, &\text{otherwise}.
\end{cases}
\end{equation*}
Here $\left[a,b,Nc\right]$ is the integral binary quadratic form
corresponding to $\delta$, and $n$ is any integer prime to $\Delta$ represented by one of the quadratic forms $[N_1a,b,N_2c]$ with $N_1N_2 = N$ and $N_1, N_2 > 0$.

The function $\chi_{\Delta}$ is invariant under the action of $\G_0(N)$ and under the action of all Atkin-Lehner involutions.
It can be computed by the following formula \cite[Section I.2, Proposition 1]{GKZ}: If $\Delta=\Delta_1\Delta_2$ is a factorization of $\Delta$ into discriminants and $N=N_1N_2$ is a factorization of $N$ into positive factors such that $(\Delta_1,N_1a)=(\Delta_2,N_2c)=1$, then
\begin{equation*}
 \chi_{\Delta}(\left[a,b,Nc\right])=\left(\frac{\Delta_1}{N_1a}\right)\left(\frac{
\Delta_2}{N_2c}\right).
\end{equation*}
If no such factorizations of $\Delta$ and $N$ exist, we have $\chi_{\Delta}(\left[a,b,Nc\right])=0$.

Since $\chi_{\Delta}(\delta)$ depends only on $\delta \in L'$ modulo $\Delta L$, we can view it as a function on the discriminant group $\dgdelta$.

We now let
\begin{equation}
 \varphi^0_{\Delta}(\lambda,z)
 =p_z(\lambda)e^{-2\pi R(\lambda,z)/|\Delta|},
\end{equation}
where $p_z(\lambda)=(\lambda,\lambda(z))$ and
$R(\lambda,z):=\frac{1}{2}(\lambda,\lambda(z))^2-(\lambda,\lambda)$. This function was recently studied extensively by H\"ovel \cite{Hoevel}. From now on, if $\Delta=1$, we omit the index $\Delta$ and simply write $\varphi^0(\lambda,z)$.
 Let
 $\varphi(\lambda,\tau,z)=e^{2\pi i Q_\Delta(\lambda)\tau} \varphi^0_{\Delta}(\sqrt{v}\lambda,z)$ (for notational purposes we drop the dependence on $
Delta$). By $\pi$ we denote the canonical projection $\pi: \dgdelta\rightarrow \dg$.

Moreover, we let $\tilde\rho=\rho$, if $\Delta>0$, and $\tilde\rho=\bar\rho$, if $\Delta<0$. 

\begin{theorem}
 The theta function
\begin{equation}
 \thetaL{\varphi} :=v^{1/2} \sum_{h \in \dg}   \sum\limits_{\substack{\delta\in \dgdelta\\ \pi(\delta) = rh \\Q_\Delta(\delta)\equiv\sgn(\Delta)Q(h)\, (\Z)}} \chi_{\Delta}(\delta) \sum\limits_{\lambda\in \Delta L+\delta} \varphi (\lambda,\tau, z)\mathfrak{e}_h
\end{equation}
is a nonholomorphic $\C[\dg]$-valued modular form of weight
$1/2$ for the representation $\widetilde{\rho}$ in the variable $\tau$.
Furthermore, it is a nonholomorphic automorphic form of weight 0 for $\G_0(N)$ in the variable $z \in D$.
\end{theorem}
\begin{proof}
 This follows from \cite[Satz 2.8]{Hoevel} and the results in \cite{AE}.
\end{proof}

We use the following representation for $\thetaL{\varphi}$ as a  Poincar\'{e} series using the lattice $K$.
We let $\epsilon =1$, when $\Delta>0$, and $\epsilon =i$, when $\Delta<0$. The following proposition can be found in \cite[Satz 2.22]{Hoevel}.

\begin{proposition}\label{twistsmallK}
We have
\begin{align*}
 & \thetaL{\varphi}= -\frac{N y^2 \bar{\epsilon}}{2i }\sum_{n=1}^\infty n\left(\frac{\Delta}{n}\right)
\\
&\quad\times \sum_{\gamma\in \widetilde{\G}_\infty\setminus\widetilde{\G}} \left[ \frac{1}{v^{1/2}} e\left(-\frac{Nn^2y^2}{2i\abs{\Delta}v}\right)\sum_{\lambda \in K'}e\left(\frac{\lambda^2}{2}\abs{\Delta}\bar{\tau}-2nN\lambda x\right)\mathfrak{e}_{r\lambda} \right]\left.\right|_{1/2,\widetilde{\rho_K}}\gamma.
\end{align*}
\end{proposition}

Now we define the theta kernel of the Shintani lift. Recall that for a lattice element $\lambda\in L'/L$ we write $\lambda=\left(\begin{smallmatrix} b/2N&-a/N\\c&-b/2N\end{smallmatrix}\right)$.
 Let
 \[
 \varphi_{\text{Sh}}(\lambda,\tau,z)= -\frac{cN\bar{z}^2-b\bar{z}+a}{4Ny^2} e^{-2\pi v R(\lambda,z)/\abs{\Delta}} e^{2\pi i Q_\Delta(\lambda)\tau}.
 \]
The Shintani theta function then transforms as follows.
\begin{theorem}
  The theta function
\begin{equation}
\thetaL{\phish} =v^{1/2} \sum_{h \in \dg}   \sum\limits_{\substack{\delta\in \dgdelta\\ \pi(\delta) = rh \\Q_\Delta(\delta)\equiv\sgn(\Delta)Q(h)\, (\Z)}} \chi_{\Delta}(\delta) \sum\limits_{\lambda\in \Delta L+\delta} \varphi_{\text{Sh}} (\lambda,\tau, z)\mathfrak{e}_h
\end{equation}
is a nonholomorphic automorphic form of weight 2 for $\G_0(N)$ in the variable $z \in D$. Moreover, $\overline{  \Theta_{\Delta,r,h}(\tau,z,\phish)}$ is a nonholomorphic $\C[\dg]$-valued modular form of weight
$3/2$ for the representation $\overline{\,\widetilde{\rho}\,\,}$ in the variable $\tau$.
\end{theorem}
\begin{proof}
 This follows from the results in \cite[p. 142]{BrdGZa08} and the results in \cite{AE}.
\end{proof}

We have the following relation between the two theta functions. This was already investigated in \cite{BF} and \cite{BriKaVia}.
\begin{lemma}\label{lm:dgl}
We have
\[
\xi_{1/2,\tau} \thetaL{\varphi}=4i\sqrt{N} y^2\frac{\partial}{\partial z} \overline{\thetaL{\phish}}.
\]
\end{lemma}
\begin{proof}
We first compute
\begin{align*}
\xi_{1/2,\tau} v^{1/2}\varphi(\lambda,\tau,z) = -v^{1/2} p_z(\lambda)e^{-2\pi v R(\lambda,z)/\abs{\Delta}}e(-Q_\Delta(\lambda)\bar{\tau}) \left(1-2\pi v \frac{R(\lambda,z)}{\abs{\Delta}}\right).
\end{align*}

For the derivative of complex conjugate of the Shintanti theta kernel we obtain
\begin{align*}
&-\frac{1}{4N}v^{1/2} e^{-2\pi v R(\lambda,z)/\abs{\Delta}}e(-Q_\Delta(\lambda)\bar{\tau})
\\
&\quad\quad\times\left(\frac{\partial}{\partial z} y^{-2}(cNz^2-bz+a)+y^{-2}(cNz^2-bz+a)(-2\pi  v)\frac{1}{\abs{\Delta}}\frac{\partial}{\partial z} R(\lambda,z)\right)
\\
&\quad=\frac{i}{ 4\sqrt{N}y^2}v^{1/2}p_z(\lambda) e^{-2\pi v R(\lambda,z)/\abs{\Delta}}e(-Q_\Delta(\lambda)\bar{\tau}) \left(1-2\pi v \frac{R(\lambda,z)}{\abs{\Delta}}\right),
\end{align*}
using that
\begin{align*}
&\frac{\partial}{\partial z} y^{-2}(cNz^2-bz+a)=-i\sqrt{N}y^{-2}p_z(\lambda),\\
&\frac{\partial}{\partial z} R(\lambda,z) =-\frac{i}{2\sqrt{N}}y^{-2}p_z(\lambda)(cN\bar{z}^2-b\bar{z}+a),\\
& y^{-2}(cNz^2-bz+a)(cN\bar{z}^2-b\bar{z}+a) = 2N R(\lambda,z).
\end{align*}
\end{proof}

\section{Theta lifts of harmonic Maass forms}\label{sec:lifts}

\label{ThetaLifts}
Recall that $\Delta$ is a fundamental discriminant and that $r\in\Z$ is such that $r^2\equiv \Delta\pmod{4N}$.
Let $F$ be a harmonic Maass  form in $H^{+}_{0}(N)$. We define the twisted theta lift of $F$ as follows
\[
 \mathcal{I}_{\Delta,r}(\tau,F)=\int_M F(z) \thetaL{\varphi}d\mu(z).
\]

\begin{theorem}\label{thm:proplift}
Let $\Delta\neq 1$ and let $F$ be a harmonic Maass form in $H^{+}_{0}(N)$ with vanishing constant term at all cusps. Then $\mathcal{I}_{\Delta,r}(\tau,F)$ is a harmonic Maass form of weight $1/2$ transforming with respect to the representation $\tilde{\rho}$. Moreover, the theta lift is equivariant with respect to the action of $O(L'/L)$.
\end{theorem}

To prove the theorem we establish a couple of results. Note that the transformation properties of the twisted theta function $\thetaL{\varphi}$ directly imply that the lift transforms with representation $\widetilde{\rho}$. The equivariance follows from \cite[Proposition 2.7]{Hoevel}. First we show that the lift is annihilated by the Laplace operator. Together with a result relating this theta lift to the Shintani lift, these results imply Theorem \ref{thm:proplift}. We also compute the lift of Poincar\'{e} series and the constant function since this will be useful in Section \ref{GeneralTheorem}. Further properties of this lift will be investigated in a forthcoming paper \cite{Claudia}.

\begin{proposition}\label{prop:basiclift}
 Let $F$ be a harmonic Maass  form in $H^{+}_{0}(N)$. Then $\mathcal{I}_{\Delta,r}(\tau,F)$ is well-defined and
\[
\Delta_{1/2,\tau} \,\mathcal{I}_{\Delta,r}(\tau,F)=0.
\]
\end{proposition}
\begin{proof}
 We first investigate the growth of the theta function $\thetaL{\varphi}=\sum_{h\in L'/L} \theta_h(\tau,z,\varphi)$ in the cusps of $M$. For simplicity we let $\Delta=N=1$. Then $L=\Z^3$ and $h=\left(\begin{smallmatrix}h'&0\\0&h'\end{smallmatrix}\right)$ with $h'=0$ or $h'=1/2$. So we consider
\[
 \theta_h(\tau, z,\varphi)=\sum_{\substack{a,c\in\Z\\b\in\Z+h'}}-\frac{v}{y}(c|z|^2-bx+a)e^{-\frac{\pi v}{y}(c|z|^2-bx+a)^2}e^{2\pi i \bar{\tau}(-b^2/4+ac)}.
\]
We apply Poisson summation on the sum over $a$. We consider the summands as a function of $a$ and compute the Fourier transform, i.e.
\begin{align*}
 &-\int_{-\infty}^\infty \frac{v}{y}(c|z|^2-bx+a)e^{-\frac{\pi v}{y}(c|z|^2-bx+a)^2}e^{2\pi i \bar{\tau}(-b^2/4+ac)}e^{2\pi i wa} da
\\
&= -y e^{-\pi i \bar{\tau}b^2/2}e^{2\pi i (c\bar{\tau}+w)(bx-c|z|^2)}
\,\, \int_{-\infty}^\infty t e^{-\pi t^2} e^{2\pi i t\frac{y}{\sqrt{v}} (c\bar{\tau}+w)}dt,
\end{align*}
where we set $t=\frac{\sqrt{v}}{y}(c|z|^2-bx+a)$. Since the Fourier transform of $xe^{-\pi x^2}$ is $ixe^{-\pi x^2}$ this equals
\begin{align*}
 &-i\frac{y^2}{\sqrt{v}} e^{-\pi i \bar{\tau}b^2/2}e^{2\pi i (c\bar{\tau}+w)(bx-c|z|^2)}(c\bar{\tau}+w)e^{-\frac{\pi y^2}{v}(c\bar{\tau}+w)^2}
\\
&=-i\frac{y^2}{\sqrt{v}} (c\bar{\tau}+w) e^{-2\pi i \bar{\tau}(b/2-cx)^2}e^{2\pi i (bxw-cx^2w)} e^{-\frac{\pi y^2}{v}|c\tau+w|^2}.
\end{align*}
We obtain that
\[
 \theta_h(\tau, z,\varphi)=-\frac{y^2}{\sqrt{v}} \sum_{\substack{w,c\in\Z\\b\in\Z+h'}} (c\bar{\tau}+w)e^{-2\pi i \bar{\tau}(b/2-cx)^2}e^{2\pi i (bxw-cx^2w)} e^{-\frac{\pi y^2}{v}|c\tau+w|^2}.
\]
If $c$ and $w$ are non-zero this decays exponentially, and if $c=w=0$ it vanishes.

In general we obtain for $h\in L'/L$ and at each cusp $\ell$
\[
 \theta_h(\tau,\sigma_\ell z,\varphi)=O(e^{-Cy^2}), \quad \text{as }y\rightarrow \infty,
\]
uniformly in $x$, for some constant $C>0$.

Thus, the growth of  $\thetaL{\varphi}$ offsets the growth of $F$ and the integral converges. By \cite[Proposition 3.10]{Hoevel} we have
\begin{align*}
 \Delta_{1/2,\tau}  \mathcal{I}_{\Delta,r}(\tau,F)&= \int_M F(z) \Delta_{1/2,\tau}\thetaL{\varphi}d\mu(z)
\\
&=\frac14  \int_M F(z) \Delta_{0,z}\thetaL{\varphi}d\mu(z).
\end{align*}
By the rapid decay of the theta function we may move the Laplacian to $F$. Since $F\in H_0^+(N)$ we have $\Delta_{0,z}F=0$, which implies the vanishing of the integral.
\end{proof}

%%%%%%%%%%%%%%%%%%%%%%%%%%%%%%%%%%%%%%%%%%%%%%%%%%%%%%%%%%%%%%%%%%%%%%%%%%%%%%%%%%%%%%%%%%%%%%%%%%%
%%%%%%%%%%%%%%%%%%%%%%%%%%%%%%%%%%%%%%%%%%%%%%%%%%%%%%%%%%%%%%%%%%%%%%%%%%%%%%%%%%%%%%%%%%%%%%%%%%%
%\subsection{Relation to the Shintani lift}

By $\mathcal{I}^{\text{Sh}}_{\Delta,r}(\tau,G)$ we denote the Shintani lifting of a cusp form $G$ of weight $2$ for $\G_0(N)$. It is defined as
\[
\mathcal{I}^{\text{Sh}}_{\Delta,r}(\tau,G)=\int_M G(z) \overline{\thetaL{\phish}} y^2d\mu(z).
\]

We then have the following relation between the two theta lifts.
\begin{theorem}\label{thm:relshin}
Let $F\in H_0^+(N)$ with vanishing constant term at all cusps. Then we have that
\[
\xi_{1/2,\tau} \left(\mathcal{I}_{\Delta,r}(\tau,F)\right)= \frac{1}{2\sqrt{N}} \mathcal{I}^{\text{Sh}}_{\Delta,r}(\tau,\xi_{0,z}(F)).
\]

\end{theorem}
\begin{proof}
By Stokes' theorem we have that
\begin{align*}
\mathcal{I}^{\text{Sh}}_{\Delta,r}(\tau,\xi_{0,z}(F))&= \int_M \xi_0(F(z))\overline{\thetaL{\phish}}y^2d\mu(z)
\\
&=- \int_M \overline{F(z)} \xi_{2,z}(\thetaL{\phish})d\mu(z)+ \lim_{t\rightarrow\infty}\int_{\partial \mathcal{F}_t} \overline {F(z)\thetaL{\phish}} d\bar{z},
\end{align*}
where $\mathcal{F}_t=\left\{z\in\mathbb{H}\,:\, \Im(z)\leq t\right\}$ denotes the truncated fundamental domain. 
Lemma \ref{lm:dgl} implies that
\begin{align*}
&- \int_M \overline{F(z)} \xi_{2,z}(\thetaL{\phish})d\mu(z)
\\
&= \frac{1}{2\sqrt{N}} \int_M \overline{F(z)} \xi_{1/2,\tau}(\thetaL{\varphi})d\mu(z)= \frac{1}{2\sqrt{N}}\xi_{1/2,\tau} \left(\mathcal{I}_{\Delta,r}(\tau,F)\right).
\end{align*}
It remains to show that 
\[
 \lim_{t\rightarrow\infty}\int_{\partial \mathcal{F}_t} \overline {F(z)\thetaL{\phish}} d\bar{z}=0.
\]
As in the proof of Proposition \ref{prop:basiclift} we have to investigate the growth of the theta function in the cusps. We have (again, $\Delta=N=1$, $L=\Z^3$, and $h'=0,1/2$)
\[
 \thetaL{\phish}=\sum_{\substack{a,c\in\Z\\b\in\Z+h'}} -\frac{c\bar{z}^2-b\bar{z}+a}{4y^2}e^{-\frac{\pi v}{y^2}(c|z|^2-bx+a)}e^{2\pi i \bar{\tau}(-b^2/4+ac)},
\]
and apply Poisson summation to the sum on $a$. Thus, we consider
\[
 \int_{-\infty}^\infty -\frac{c\bar{z}^2-b\bar{z}+a}{4y^2}e^{-\frac{\pi v}{y^2}(c|z|^2-bx+a)}e^{2\pi i \bar{\tau}(-b^2/4+ac)} e^{2\pi i wa}da.
\]
Proceeding as before, we obtain
\begin{align*}
 \theta_h(\tau,z,\varphi_{\text{Sh}})= -\frac{1}{4\sqrt{v}y}& \sum_{\substack{w,c\in\Z\\b\in\Z+h'}} e^{-2\pi i \bar{\tau}(b/2-cx)^2}e^{2\pi i (bxw-cx^2w)}
\\
&\quad\quad\times\left(c\bar{z}^2+biy-c|z|^2+i\frac{y^2}{v}(c\bar{\tau}+w)\right)e^{-\frac{\pi y^2}{v}|c\tau+w|^2} .
\end{align*}
If $c$ and $w$ are not both equal to $0$ this vanishes in the limit as $y\rightarrow \infty$. In this case, the whole integral vanishes.
But if $c=w=0$ we have
\[
 -\frac{i}{4\sqrt{v}}\sum_{b\in\Z+h'} b e^{\pi i \bar{\tau}b^2/2}.
\]
Thus, we are left with (the complex conjugate of)
\[
 \int_{\partial \mathcal{F}_T} F(z)\thetaL{\phish}dz= \frac{i}{4\sqrt{v}}\sum_{b\in\Z+h'} b e^{\pi i \bar{\tau}b^2/2} \int_{0}^1 F(x+iT)dx.
\]
We see that 
\[
 \lim_{T\rightarrow\infty}\int_{0}^1 F(x+iT)dx=0,
\]
since the constant coefficient of $F$ vanishes.
Therefore, 
\[
  \lim_{T\rightarrow\infty}\int_{\partial M_T}\overline{ F(z)\thetaL{\phish}}d\bar{z}= 0.%a(0)\frac{i}{4\sqrt{v}}\sum_{b\in\Z+h'} b e^{\pi i \bar{\tau}b^2/2}.
\]
Generalizing to arbitrary $N$, a similar result holds for the other cusps of $M$.

\end{proof}

For a cusp form $G=\sum_{n=1}^\infty b(n)q^n\in S_2^{\text{new}}(N)$ we let $L(G,\Delta,s)$ be its twisted $L$-function
\[
 L(G,\Delta,s)=\sum_{n=1}^\infty \left(\frac{\Delta}{n}\right)b(n)n^{-s}.
\]

The relation to the Shintani lifting directly implies
\begin{proposition}\label{prop:whtowh} 
Let $F\in H_0^+(N)$ with vanishing constant term at all cusps and let $\xi_{0,z}(F)=F_E\in S_2^{\text{new}}(N)$. 
The lift  
$
 \calI_{\Delta,r}(\tau,F)
$
is weakly holomorphic if and only if 
\[
 L(F_E, \Delta,1)=0.
\] 
In particular, this happens if $F$ is weakly holomorphic.
\end{proposition}
\begin{proof}
 Clearly, the lift is weakly holomorphic if and only if the Shintani lifiting of $F_E$ vanishes. This is trivially the case when $F_E=\xi_0(F)=0$, i.e. when $F$ is weakly holomorphic. 
 In the other case, the coefficients of the Shintani lifting are given by (in terms of Jacobi forms; for the definition of Jacobi forms and the cycle integral $r$ see \cite{GKZ})
\[
 \mathcal{I}^{\text{Sh}}_{\Delta,r}(\tau,\xi_{0,z}(F))=\sum_{\substack{n,r_0\in\Z\\r_0^2<4nN}}r_{1,N,\Delta(r_0^2-4nN),rr_0,\Delta}(F_E)q^n\zeta^{r_0}.
\]
Now by the Theorem and Corollary in Section II.4 in \cite{GKZ} we have
\[
 |r_{1,N,\Delta(r_0^2-4nN),rr_0,\Delta}(F_E)|^2=\frac{1}{4\pi^2}\abs{\Delta}^{1/2} \abs{r_0^2-4nN}^{1/2}\, L(F_E,\Delta,1)\, L(F_E,r_0^2-4nN,1).
\]
Since $r_0$ and $n$ vary this expression vanishes if and only if $ L(F_E,\Delta,1)$ vanishes.
\end{proof}

\begin{proof}[Proof of Theorem \ref{thm:proplift}]
 Proposition \ref{prop:basiclift} implies that an $F\in H^{+}_0(N)$ with vanishing constant term at all cusps maps to a form of weight $1/2$ transforming with representation $\tilde{\rho}$ that is annihilated by the Laplace operator $\Delta_{1/2,\tau}$.
 Theorem \ref{thm:relshin} then implies, that the lift satisfies the correct growth conditions at all cusps.
\end{proof}

%%%%%%%%%%%%%%%%%%%%%%%%%%%%%%%%%%%%%%%%%%%%%%%%%%%%%%%%%%%%%%%%%%%%%%%%%%%%%%%%%%%%%%%%%%%%%%%%%%%
%%%%%%%%%%%%%%%%%%%%%%%%%%%%%%%%%%%%%%%%%%%%%%%%%%%%%%%%%%%%%%%%%%%%%%%%%%%%%%%%%%%%%%%%%%%%%%%%%%%

\subsection{Fourier expansion of the holomorphic part}

Now we turn to the computation of the Fourier coefficients of positive index of the holomorphic part of the theta lift.

Let $h\in L'/L$ and $m\in\Q_{>0}$ with $m \equiv \sgn(\Delta)Q(h)\ (\Z)$.
We define a twisted Heegner divisor on $M$ by
\[
Z_{\Delta,r}(m,h)= \sum\limits_{\lambda \in \G_0(N) \backslash L_{rh,m\abs{\Delta}}}\frac{\chi_{\Delta}(\lambda)}{\left|\overline\G_{\lambda}\right|} Z(\lambda).
\]
Here $\overline{\G}_\lambda$ denotes the stabilizer of $\lambda$ in $\overline{\G_0(N)}$.

Let $F$ be a harmonic Maass  form of weight $0$ in $H^{+}_0(N)$.
Then the twisted modular trace function is defined as follows
\begin{equation}\label{def:trace11}
\tr_{\Delta,r}(F;m,h) = \sum\limits_{z\in Z_{\Delta,r}(m,h)}F(z)= \sum_{\lambda\in \G\setminus L_{\abs{\Delta}m},rh}\frac{\chi_\Delta(\lambda)}{\abs{\bar\G_\lambda}} f(D_\lambda) .
\end{equation}
Here we need to define a refined modular trace function. We let
\[
 L_{\abs{\Delta}m,rh}^+=\left\{\lambda=\begin{pmatrix}
                                         \frac{b}{2N}&-\frac{a}{N}\\c&-\frac{b}{2N}
                                        \end{pmatrix}\in L_{\abs{\Delta}m,rh}\,;\, a\geq 0\right\},
\]
and similarly
\[
 L_{\abs{\Delta}m,rh}^-=\left\{\lambda=\begin{pmatrix}
                                         \frac{b}{2N}&-\frac{a}{N}\\c&-\frac{b}{2N}
                                        \end{pmatrix}\in L_{\abs{\Delta}m,rh}\,;\, -a>0\right\},
\]
and define modular trace functions
\[
 \tr^+_{\Delta,r}(F;m,h)=\sum_{\lambda\in \G\setminus L^+_{\abs{\Delta}m,rh}}\frac{\chi_\Delta(\lambda)}{\abs{\bar\G_\lambda}} f(D_\lambda)
\]
and
\[
 \tr^-_{\Delta,r}(F;m,h)=\sum_{\lambda\in \G\setminus L^-_{\abs{\Delta}m,rh}}\frac{\sgn(\Delta)\chi_\Delta(\lambda)}{\abs{\bar\G_\lambda}} f(D_\lambda).
\]

\begin{comment}
Recall, that there is an involution acting on harmonic Maass forms via
\[
 F(z)\mapsto \overline{F(-\bar{z})}.
\]
\end{comment}
\begin{theorem}\label{thm:trace}
Let $F$ be a harmonic Maass  form of weight $0$ in $H^{+}_0(N)$, $m>0$, and $h\in L'/L$. 
The coefficients of index $(m,h)$ of the holomorphic part of the lift $\calI_{\Delta,r}(\tau,F)$ are given by 
\begin{equation}\label{eq:trace}
  \frac{\sqrt{\Delta}}{2\sqrt{m}}\left(\tr^+_{\Delta,r}(F;m,h)-\tr^-_{\Delta,r}(F;m,h)\right) .
\end{equation}
\end{theorem}

\begin{proof}
To ease notation we start proving the result when $\Delta=1$. Using the arguments of the proof of Theorem 5.5 in \cite{AE} it is straightforward to later generalize to the case $\Delta\neq 1$.

We consider the Fourier expansion of $\int_M F(z) \Theta(\tau,z,\varphi)d\mu(z)$, namely
\begin{align}\label{proof:fc12}
\sum_{h\in L'/L}\sum_{m\in\Q}\left(\sum_{\lambda\in L_{m,h}}\int_M F(z) v^{1/2} \varphi^0(\sqrt{v}\lambda,z)d\mu(z) \right)e^{2\pi im\tau}.
\end{align}
We denote the $(m,h)$-th coefficient of the holomorphic part of \eqref{proof:fc12} by $C(m,h)$.
Using the usual unfolding argument implies that
\begin{align*}
 C(m,h)&=\sum_{\lambda\in \G\setminus L_{m,h}}\frac{1}{|\bar\G_\lambda|}\int_{D} F(z) v^{1/2} \varphi^0(\sqrt{v}\lambda,z)d\mu(z)
\\
&=\sum_{\lambda\in \G\setminus L^+_{m,h}}\frac{1}{|\bar\G_\lambda|}\int_{D} F(z) v^{1/2} \varphi^0(\sqrt{v}\lambda,z)d\mu(z)
\\
&\quad+\sum_{\lambda\in \G\setminus L^-_{m,h}}\frac{1}{|\bar\G_\lambda|}\int_{D} F(z) v^{1/2} \varphi^0(\sqrt{v}\lambda,z)d\mu(z).
\end{align*}
Since $ \varphi^0(-\sqrt{v}\lambda,z)=- \varphi^0(\sqrt{v}\lambda,z)$ the latter summand equals
\[
- \sum_{\lambda\in \G\setminus L^-_{m,h}}\frac{1}{|\bar\G_{-\lambda}|}\int_{D} F(z) v^{1/2} \varphi^0(-\sqrt{v}\lambda,z)d\mu(z).
\]
As in \cite{KS} and \cite{BruinierOno2} we rewrite the integral over $D$ as an integral over $G(\R)=\mathrm{SL}_2(\R)$. We normalize the Haar measure such that the maximal compact subgroup $\mathrm{SO}(2)$ has volume $1$. We then have 
\[
 \int_{D} F(z) \varphi^0(\sqrt{v}\lambda,z)d\mu(z)=\int_{G(\R)} F(gi) \varphi^0(\pm \sqrt{v}\lambda,gi)dg, \quad\text{ for } \lambda \in \G\setminus L^{\pm}_{m,h}.
\]
Note that in \cite{KS} it is assumed that $\SL_2(\R)$ acts transitively on vectors of the same norm. This is not true. However, $\SL_2(\R)$ acts transitively on vectors of the same norm satisfying $a>0$. Therefore, there is a $g_1\in \mathrm{SL}_2(\R)$ such that $g_1^{-1}\textbf{.}\lambda=\sqrt{m} \lambda(i)$ for $\lambda\in L_{m,h}^+$. Similarly, there is a $g_1\in \mathrm{SL}_2(\R)$ such that $g_1^{-1}\textbf{.}(-\lambda)=\sqrt{m} \lambda(i)$ for $\lambda\in L_{m,h}^-$. So, we have
\begin{align*}
 C(m,h)&= \sum_{\lambda\in \G\setminus L^+_{m,h}}\frac{1}{|\bar\G_\lambda|} v^{1/2} \int_{G(\R)} F(gg_1i) \varphi^0\left(\sqrt{v}\sqrt{m}g^{-1}\textbf{.}\lambda(i),i\right)dg 
\\
&\quad -\sum_{\lambda\in \G\setminus L^-_{m,h}}\frac{1}{|\bar\G_{-\lambda}|} v^{1/2}\int_{G(\R)} F(gg_1i) \varphi^0\left(\sqrt{v}\sqrt{m}g^{-1}\textbf{.}\lambda(i),i\right)dg.
\end{align*}
Using the Cartan decomposition of $\SL_2(\R)$ we find proceeding as in \cite{KS} that
\begin{equation}
C(m,h)= \sum_{\lambda\in \G\setminus L^+_{m,h}}\frac{1}{|\bar\G_\lambda|} F(D_\lambda) v^{1/2} Y(\sqrt{mv})-
\sum_{\lambda\in \G\setminus L^-_{m,h}}\frac{1}{|\bar\G_{-\lambda}|}F(D_{-\lambda}) v^{1/2} Y(\sqrt{mv}),
\end{equation}
where
\begin{equation}\label{proof:fc13}
 Y(t)=4\pi \int_1^\infty \varphi^0(t\alpha(a)^{-1} \textbf{.} \lambda(i),i)\frac{a^2-a^{-2}}{2}\frac{da}{a}.
\end{equation}
Here $\alpha(a)=\left(\begin{smallmatrix}a&0\\0&a^{-1}\end{smallmatrix}\right)$. We have that
\[
  \varphi^0(t\alpha(a)^{-1} \textbf{.} \lambda(i),i)=t(a^2+a^{-2}) e^{-\pi t^2(a^2-a^{-2})^2}.
\]
Substituting $a=e^{r/2}$ we obtain that \eqref{proof:fc13} equals
\begin{align*}
 4\pi t \int_0^\infty \cosh (r) \sinh(r) e^{-4\pi t^2\sinh(r)^2}dr
= \frac{1}{2t}.
\end{align*}
Thus, we have $Y(\sqrt{mv})=\frac{1}{2\sqrt{mv}}$ which implies that
\[
 C(m,h)= \frac{1}{2\sqrt{m}} \left(  \sum_{\lambda\in \G\setminus L^+_{m,h}}\frac{1}{|\bar\G_\lambda|} F(D_\lambda)- \sum_{\lambda\in \G\setminus L^-_{m,h}}\frac{1}{|\bar\G_\lambda|} F(D_\lambda)\right),
\]
since $|\bar\G_\lambda|=|\bar\G_{-\lambda}|$ and $D_\lambda=D_{-\lambda}$.

Using the methods of \cite{AE} it is not hard to see that the $(m,h)$-th coefficient of the twisted lift is equal to
\[
 \frac{\sqrt{\Delta}}{2\sqrt{m}} \left(  \sum_{\lambda\in \G\setminus L^+_{m\abs{\Delta},rh}}\frac{\chi_\Delta(\lambda)}{|\bar\G_\lambda|} F(D_\lambda)- \sum_{\lambda\in \G\setminus L^-_{m\abs{\Delta},rh}}\frac{\chi_\Delta(-\lambda)}{|\bar\G_\lambda|} F(D_\lambda)\right).
\]
We have that $\chi_\Delta(-\lambda)=\sgn(\Delta)\chi_\Delta(\lambda)$ which implies the result. 

\end{proof}

\subsection{Lift of Poincar\'{e} series and constants}
In this section we compute the lift of Poincar\'{e} series and the constant function in the case $\Delta\neq 1$. This will be useful for the computation of the principal part of the theta lift.
\begin{theorem}\label{thm:liftpoincare}
We have
\begin{align*}
\mathcal{I}_{\Delta,r}(\tau,F_m(z,s,0))= \frac{2^{-s+1}i}{\G(s/2)}\, \sqrt{\pi N\abs{\Delta}}\bar{\epsilon}\,
 \sum_{n|m} \left(\frac{\Delta}{n}\right)\mathcal{F}_{\frac{m^2}{4Nn^2}\abs{\Delta},-\frac{m}{n}r}\left(\tau,\frac{s}{2}+\frac{1}{4},\frac12\right).
\end{align*}
\end{theorem}
\begin{remark}
In particular, we have
\[
\mathcal{I}_{\Delta,r}(\tau,F_m(z,1,0))= i\bar{\epsilon} \sqrt{N\abs{\Delta}}\,
 \sum_{n|m} \left(\frac{\Delta}{n}\right)\mathcal{F}_{\frac{m^2}{4Nn^2}\abs{\Delta},-\frac{m}{n}r}\left(\tau,\frac{3}{4},\frac12\right).
\]
\end{remark}

\begin{proof}
The proof follows the one in \cite[Theorem 3.3]{BruinierOno} or \cite[Theorem 4.3]{Alfes}.
Using the definition of the Poincar\'{e} series (\ref{def:poincare}) and an unfolding argument we obtain
\begin{align*}
\mathcal{I}_{\Delta,r}(\tau,F_m(z,s,0))= \frac{1}{\G(2s)}\int_{\G_\infty\setminus\h}\M_{s,0}(4\pi m y)e(-mx)\thetaL{\varphi}d\mu(z).
\end{align*}
By Proposition \ref{twistsmallK} this equals
\begin{align*}
-\frac{\bar{\epsilon}\,N}{\G(2s)2i}
  \sum\limits_{n=1}^\infty \left(\frac{\Delta}{n}\right)n\sum\limits_{\gamma\in\widetilde{\G}_\infty\setminus\widetilde{\G}} I(\tau,s,m,n)|_{1/2,\widetilde{\rho}_K} \ \gamma,
\end{align*}
where
\begin{align*}
  I(\tau,s,m,n)&= \int_{y=0}^\infty\int_{x=0}^1 y^2 \M_{s,0}(4\pi m y)e(-mx) \exp\left(-\frac{\pi n^2Ny^2}{\abs{\Delta}v}\right)
  \\
  & \quad\quad\quad\quad\quad\times v^{-1/2}\sum_ {\lambda\in K'}e\left(\abs{\Delta}Q(\lambda)\bar{\tau}-2N\lambda n x\right)\mathfrak{e}_{r\lambda}\frac{dxdy}{y^2}.
\end{align*}
Identifying $K'=\Z\left(\begin{smallmatrix} 1/2N&0\\0&-1/2N\end{smallmatrix}\right)$ we find that
\[
 \sum_ {\lambda\in K'}e\left(\abs{\Delta}Q(\lambda)\bar{\tau}-2N\lambda n x\right)\mathfrak{e}_{r\lambda}=\sum\limits_{b\in\Z}e\left(-\abs{\Delta}\frac{b^2}{4N}\bar{\tau}-nbx\right)\mathfrak{e}_{rb}.
\]
Inserting this in the formula for $I(\tau, s,m,n)$, and integrating over $x$, we see that $I(\tau,s,m,n)$ vanishes whenever $n\nmid m$ and the only summand occurs for $b=-m/n$, when $n\mid m$. Thus, $ I(\tau,s,m,n)$ equals
\begin{align}\label{proof:pc222}
v^{-1/2} e\left(-\abs{\Delta}\frac{m^2}{4Nn^2}\bar{\tau}\right)\ \cdot\ \int_{y=0}^\infty \M_{s,0}(4\pi m y) \exp\left(-\frac{\pi n^2Ny^2}{\abs{\Delta}v}\right)dy  \,\, \mathfrak{e}_{-rm/n}.
 \end{align}
To evaluate the integral in \eqref{proof:pc222} note that (see for example (13.6.3) in \cite{Pocket})
\[
\M_{s,0}(4\pi m y)=2^{2s-1}\G\left(s+\frac12\right)\sqrt{4\pi m y}\cdot I_{s-1/2}(2\pi m y).
\]
Substituting $t=y^2$ yields
\begin{align*}
& \int_{y=0}^\infty \M_{s,0}(4\pi m y) \exp\left(-\frac{\pi n^2Ny^2}{\abs{\Delta}v}\right)dy
\\
&\notag = 2^{2s-1}\G\left(s+\frac12\right) \int_{y=0}^\infty  \sqrt{4\pi m y} \ I_{s-1/2}(2\pi m y)  \exp\left(-\frac{\pi n^2Ny^2}{\abs{\Delta}v}\right)dy
\\
&\notag = 2^{2s-1}\G\left(s+\frac12\right)  \sqrt{m\pi} \int_{t=0}^\infty t^{-1/4}  I_{s-1/2}(2\pi m t^{1/2})  \exp\left(-\frac{\pi n^2Nt}{\abs{\Delta}v}\right)dt.
\end{align*}
The last integral is a Laplace transform and is computed in \cite{tables} (see (20) on p. 197). It equals
\[
\frac{\G\left(\frac{s}{2}+\frac12\right)}{\G\left(s+\frac12\right)}(\pi m)^{-1} \left(\frac{\pi n^2N}{\abs{\Delta}v}\right)^{-1/4} \exp\left(\frac{\pi m^2\abs{\Delta}v}{2n^2N}\right)  M_{-\frac{1}{4}, \frac{s}{2}-\frac{1}{4}}\left(\frac{\pi m^2\abs{\Delta}v}{n^2N}\right).
\]
Therefore, we have that $I(\tau,s,m,n)$ equals
\begin{align*}
2^{2s-1} \G\left(\frac{s}{2}+\frac12\right) \sqrt{\frac{\abs{\Delta}}{\pi N n^2}} e\left(-\frac{m^2\abs{\Delta}u}{4 n^2 N}\right)\mathcal{M}_{s/2+1/4,1/2}\left(\frac{\pi m^2\abs{\Delta}v}{n^2N}\right)\mathfrak{e}_{-rm/n}.
\end{align*}
Putting everything together we obtain the following for the lift of $F_m(z,s,0)$
\begin{align*}
&-\frac{2^{2s-2}\G(s/2+1/2)\bar{\epsilon}}{\G(2s)i}\sqrt{\frac{N\abs{\Delta}}{\pi}}
 \sum_{n|m} \left(\frac{\Delta}{n}\right)
\\
&\quad\times\sum\limits_{\gamma\in\widetilde{\G}_\infty\setminus\widetilde{\G}} \left[  e\left(-\frac{m^2\abs{\Delta}u}{4Nn^2}\right) \mathcal{M}_{s/2+1/4,1/2}\left(\frac{\pi m^2\abs{\Delta}v}{n^2N}\right)\mathfrak{e}_{-rm/n}\right]\left.\right|_{1/2,\widetilde{\rho}_K} \ \gamma
\\
&= - \frac{2^{-s+1}}{i\G(s/2)}\, \sqrt{\pi N\abs{\Delta}}\bar{\epsilon}\,
 \sum_{n|m} \left(\frac{\Delta}{n}\right)\mathcal{F}_{\frac{m^2}{4Nn^2}\abs{\Delta},-\frac{m}{n}r}\left(\tau,\frac{s}{2}+\frac{1}{4},\frac12\right).
\end{align*}
\end{proof}

We define
\[
 \Theta_{K}(\tau)=\sum_{\lambda\in K'}e(Q(\lambda)\tau)\mathfrak{e}_{\lambda+K}.
\]

\begin{theorem}\label{thm:liftconstant}
Let $N=1$ and $\Delta<0$ (for $\Delta>0$ and $N=1$ the lift vanishes), $\epsilon_\Delta(n)=\left(\frac{\Delta}{n}\right)$ and $L\left(\epsilon_\Delta,s\right)$ be the Dirichlet $L$-series associated with $\epsilon_\Delta$. We have
\[
  \mathcal{I}_{\Delta,r}(\tau,1)= \frac{\bar{\epsilon}\,i}{\pi}\abs{\Delta} L\left(\epsilon_\Delta,1\right)\Theta_K(\tau).
\]
\end{theorem}

\begin{proof}
This result follows analogously to \cite[Theorem 7.1, Corollary 7.2]{BrFu06} and \cite[Theorem 6.1]{AE}.
We compute the lift of the nonholomorphic weight $0$ Eisenstein series and then take residues at $s= 1/2$. Let $z\in\H$, $s\in\C$ and
\[
 \mathcal{E}_0(z,s)=\frac12\zeta^*(2s+1)\sum_{\gamma\in\G_\infty\setminus\\SL_2(\Z)} (\Im(\gamma z))^{s+\frac12},
\]
where $\zeta^*(s)$ is the completed Riemann Zeta function. The Eisenstein series $\mathcal{E}_0(z,s)$ has a simple pole at $s=\frac12$ with residue $\frac12$.
Using the standard unfolding trick we obtain
\begin{align*}
  \mathcal{I}_{\Delta,r}(\tau,\mathcal{E}_0(z,s))=\zeta^*(2s+1)\int_{\G_\infty\setminus \H}\thetaL{\varphi}y^{s+\frac12} d\mu(z).
\end{align*}
By Proposition \ref{twistsmallK} we have that this equals
\begin{align*}
& -\zeta^*(2s+1)\frac{\bar{\epsilon}}{2i}\sum_{n\geq 1}n\left(\frac{\Delta}{n}\right) \sum_{\gamma\in \widetilde{\G}_\infty\setminus\widetilde{\G}}\phi(\tau)^{-1}\tilde{\rho}^{-1}_K(\gamma)\frac{1}{\Im(\gamma\tau)^{1/2}}
\\
&\quad\times \int_{y=0}^\infty y^{s+\frac12}\exp\left(-\frac{\pi n^2y^2}{\abs{\Delta}\Im(\gamma\tau)}\right)dy
\\
&\quad\times \int_{x=0}^1\sum_{\lambda\in K'}e\left(\frac{\lambda^2\bar{\tau}}{2\abs{\Delta}}-2\lambda nx\right)\mathfrak{e}_{r\lambda}dx.
\end{align*}
The integral over $x$ equals $\mathfrak{e}_0$ and the one over $y$ equals
\[
 \frac12 \G\left(\frac{s}{2}+\frac34\right) (\abs{\Delta}\Im(\gamma\tau))^{\frac{s}{2}+\frac34}\pi^{-\frac{s}{2}-\frac34} n^{-s-\frac32}.
\]
Thus, we have
\begin{align*}
  \mathcal{I}_{\Delta,r}(\tau,\mathcal{E}_0(z,s))&=-\zeta^*(2s+1)\frac{\bar{\epsilon}}{2i}\G\left(\frac{s}{2}+\frac34\right) \abs{\Delta}^{\frac{s}{2}+\frac34} \pi^{-\frac{s}{2}-\frac34}
\\
&\quad\times L\left(\epsilon_\Delta,s+\frac12\right)\,\,\frac12\,\sum_{\gamma\in \widetilde{\G}_\infty\setminus\widetilde{\G}} (v^{\frac12(s+\frac12)}\mathfrak{e}_0)|_{1/2,K}\gamma.
%\\
%&= -\zeta^*(2s+1)\frac{N^{-\frac{s}{2}+\frac14}\,\abs{\Delta}^{\frac{s}{2}+\frac34}\, \bar\epsilon}{4i\sqrt{\pi}}\frac{\G\left(\frac{s}{2}+\frac34\right)}%{\G\left(\frac{s}{2}+\frac14\right)} 
%\\
%&\quad\times \Lambda\left(\epsilon_\Delta,s+\frac12\right)\,\sum_{\gamma\in \tilde{\G}_\infty\setminus\tilde{\G}} (v^{\frac12(s+\frac12)}\mathfrak{e}_0)|_{1/2,K}\gamma.
\end{align*}
We now take residues at $s=1/2$ on both sides. Note that the residue of the weight $1/2$ Eisenstein series is given by (see \cite[Proof of Proposition 5.14]{BIF})
\[
 \mathrm{res}_{s=1/2}\left(\frac12\,\sum_{\gamma\in \widetilde{\G}_\infty\setminus\widetilde{\G}} (v^{\frac12(s+\frac12)}\mathfrak{e}_0)|_{1/2,K}\gamma\right)=\frac{6}{\pi} \Theta_K(\tau).
\]
We have $\zeta^*(2)=\pi/6$ which concludes the proof of the theorem.
\end{proof}

\section{General version of Theorem~\ref{thm3} and its proof}\label{GeneralTheorem}

Here we give the general version of Theorem~\ref{thm3}, give its proof, and then
conclude with some numerical examples.

%\subsection{Statement of the general version of Theorem~\ref{thm3}}
We begin with some notation. Let $L$ be the lattice of discriminant $2N$
defined in Section \ref{sec:lattice} and let $\rho=\rho_{1}$ be as in Section \ref{sec:weil}. 
Let $F_E\in
S_{2}^{new}(\Gamma_0(N_E))$ be a normalized newform of weight $2$
associated to the elliptic curve $E/\Q$. Let $\epsilon\in\left\{\pm 1\right\}$ be the eigenvalue of the Fricke involution on $F_G$. If $\epsilon =1$, we put $\rho= \bar\rho$ and assume that $\Delta$ is a negative fundamental disriminant. If $\epsilon=-1$ we put $\rho= \rho$ and assume that
$\Delta$ is a positive fundamental discriminant. There is a newform
$g_E\in S_{3/2,\rho}^{new}$ mapping to $F_E$ under the Shimura
correspondence. We may normalize $g_E$ such that all its coefficients
are contained in $\Q$.

Recall that
\[
 \whZ_E(z)= \zeta(\Lambda_E;\mathcal{E}_E(z))-S(\Lambda_E)\mathcal{E}_E(z)-\frac{\deg(\phi_E)}{4\pi ||F_E||^2}\overline{\mathcal{E}_E(z)},
\]
and $M_E(z)$ is chosen such that $\whZ_E(z)-M_E(z)$ is holomorphic on $\mathbb{H}$. By $a_{\ell,\whZ_E}(0)$ and $a_{\ell,M_E}(0)$ we denote the constant terms of these two functions at the cusp $\ell$.

We then let 
\[
 \whZ_E^*(z)=\frac{1}{\sqrt{\abs{\Delta} N}}\left(\whZ_E(z)-\sum_{\ell\in\G\setminus\mathrm{Iso}(V)}a_{\ell,\whZ_E}(0)\right).
\]
Analogously, we let
\[
  M_E^*(z)=\frac{1}{\sqrt{\abs{\Delta} N}}\left(M_E(z)-\sum_{\ell\in\G\setminus\mathrm{Iso}(V)}a_{\ell,M_E}(0))\right).
\] 

Then $\whZ_E^*(z)- M_E^*(z)$ is a harmonic Maass form of weight $0$.

By $f_{E,\Delta,r}=f_E$ we denote the twisted theta lift of $ \whZ_E^*(z)-M_E^*(z)$ as in Section \ref{sec:lifts}.

 We begin
with some notation. Let $L$ be the lattice of discriminant $2N$
defined in Section \ref{sec:lattice} and let $\rho=\rho_{1}$ be as in Section \ref{sec:weil}. Let $k\in
\frac{1}{2}\Z\setminus \Z$. The space of vector-valued holomorphic
modular forms $M_{k,\bar{\rho}}$ is isomorphic to the space of skew
holomorphic Jacobi forms $J_{k+1/2,N}^{skew}$ of weight $k+1/2$ and
index $N$. Moreover, $M_{k,\rho}$ is isomorphic to the space
of holomorphic Jacobi forms $J_{k+1/2,N}$.
The subspace $S_{k,\bar\rho}^{new}$ of newforms of the cusp forms $S_{k,\bar\rho}$ is
isomorphic as a module over the Hecke algebra to the space of
newforms $S^{new,+}_{2k-1}(\Gamma_0(N))$ of weight $2k-1$ for $\Gamma_0(N)$ on
which the Fricke involution acts by multiplication with
$(-1)^{k-1/2}$. The isomorphism is given by the Shimura
correspondence \cite{Sh}. Similarly, the subspace $S_{k,\rho}^{new}$ of
newforms of $S_{k,\rho}$ is isomorphic as a module over the
Hecke algebra to the space of newforms $S^{new,-}_{2k-1}(\Gamma_0(N))$ of
weight $2k-1$ for $\Gamma_0(N)$ on which the Fricke involution acts
by multiplication with $(-1)^{k+1/2}$ \cite{GKZ}. Let $\epsilon$ be the eigenvalue of the Fricke involution on $G$.

The Hecke $L$-series of any $G\in
S^{new,\pm}_{2k-1}(\Gamma_0(N))$ satisfies a functional equation under
$s\mapsto 2k-1-s$ with root number $-\epsilon$. If $G\in
S^{new,\pm}_{2k-1}(\Gamma_0(N))$ is a normalized newform (in particular a
common eigenform of all Hecke operators), we denote by $F_G$ the
number field generated by the Hecke eigenvalues of $G$. It is well
known that we may normalize the preimage of $G$ under the Shimura
correspondence such that all its Fourier coefficients are contained
in $F_G$.

\begin{comment}
Let $\rho$ be one of the representations $\rho_L$ or $\bar\rho_L$.
For every positive integer $l$ there is a Hecke operator $T(l)$ on
$M_{k,\rho}$ which is self adjoint with respect to the Petersson
scalar product. The action on the Fourier expansion
$g(\tau)=\sum_{h,n} b(n,h)e(n\tau)\frake_h$ of any $g\in M_{k,\rho}$
can be described explicitly.  For instance, if $p$ is a prime coprime to
$N$ and we write $g\mid_{k} T(p)= \sum_{h,n}
b^*(n,h)e(n\tau)\frake_h$, we have
\begin{align}
\label{heckeop} b^*(n,h)=b(p^2n,ph)+p^{k-3/2}\leg{4N\sigma n}{p}
b(n,h)+p^{2k-2}b(n/p^2,h/p),
\end{align}
where $\sigma=1$ if $\rho=\rho_L$, and $\sigma=-1$ if
$\rho=\bar\rho_L$. There are similar formulas for general $l$.
The Hecke operators act on harmonic Maass forms and on weakly
holomorphic modular forms in an analogous way. In particular, the
formula for the action on Fourier coefficients is the same.
\end{comment}

%Here we give the general version of Theorem~\ref{thm3}.
\begin{theorem}\label{GeneralCase}
Assume that $E/\Q$ is an elliptic curve of square free conductor $N_E$, and suppose that
$F_E|_2 W_{N_E}=\epsilon F_E$.  Denote the coefficients of $f_E(\tau)$ by
$c_E^{\pm}(h,n)$. Then the following are true:
\begin{enumerate}
\item[(i)]  If $d\neq 1$ is a fundamental discriminant and $r\in \Z$ such that
$d\equiv r^2\pmod{4N_E}$, and $\epsilon d<0$, then
$$
L(E_{d},1)=8\pi^2 ||F_E||^2 ||g_E||^2\sqrt{\frac{|d|}{N_E}}\cdot
c_E^{-}(\epsilon d,r)^2.
$$
\item[(ii)] If $d\neq 1$ is a fundamental discriminant and $r\in \Z$ such that
$d\equiv r^2\pmod{4N_E}$ and $\epsilon d>0$, then
$$
L'(E_{d},1)=0 \ \ \iff \ \
c_E^{+}(\epsilon d,r)\in \overline{\Q} \ \
\iff \ \
c_E^{+}(\epsilon d,r)\in \Q.
$$
\end{enumerate}
\end{theorem}
\begin{remark}
 In contrast to Bruinier and Ono in \cite{BruinierOno} we are able to relate the weight $1/2$ form to the elliptic curve in a direct way.
\end{remark}

\begin{proof}
To prove Theorem~\ref{GeneralCase}, we shall employ the results in Section 7 in \cite{BruinierOno}.
It suffices to prove that $f_E$
can be taken for $f$ in Theorem 7.6 and 7.8 in \cite{BruinierOno}. Therefore, we need to prove that $f_E$ has rational principal part and that $\xi_{1/2}(f_E)\in\R g$, where $g_E$ is the preimage of $F_E$ under the Shimura lift. (In the case we consider it suffices to require that $\xi_{1/2}(f)\in \R g_E$ in \cite[Theorem 7.6]{BruinierOno}.)

We first prove that $f_E$ has rational principal part at the cusp $\infty$. We write $\whZ_E^*(z)-M_E^*(z)$ as a linear combination of Poincar\'{e} series and constants, i.e.
\[
 \whZ_E^*(z)-M_E^*(z)= C+ \frac{1}{\sqrt{\abs{\Delta}N}}\sum_{m>0} a_{\whZ_E}(-m) F_m(z,1,0)+\frac{1}{\sqrt{\abs{\Delta}N}}\sum_{k>0} a_{M_E}(-k) F_k(z,1,0).
\]
Here $C$ is a constant and the coefficients $a_{\whZ_E}(-m)$ and $a_{M_E}(-k)$ are rational by construction.

Then, by Theorem~\ref{thm:liftpoincare} and Theorem~\ref{thm:liftconstant} the coefficients of the principal part of $f_E$ are rational. For the other cusps of $\G_0(N)$ this follows by the equivariance of the theta lift under $O(L'/L)$ and the fact that we can identify $O(L'/L)$ with the group generated by the Atkin-Lehner involutions.

By construction we have 
\[
 \xi_0\left(\whZ_E^*(z)-M_E^*(z)\right)= \frac{-\mathrm{deg}(\phi_E)}{\sqrt{\abs{\Delta} N}||F_E||^2}F_E.
\]
At the same time Theorem~\ref{thm:relshin} implies that
\[
 \calI^{\text{Sh}}_{\Delta,r}\left(\frac{-\mathrm{deg}(\phi_E)}{\sqrt{\abs{\Delta} N}||F_E||^2}F_E\right)= 2\sqrt{N} \xi_{1/2}(f_E).
\]
Thus, we have that $\xi_{1/2}(f_E)\in\R g_E$.
\end{proof}

\section{Examples}\label{EXAMPLES}

Here we give examples
which illustrate the results proved in this paper.

\begin{example}
For $X_0(11),$ we have a single isogeny class. The strong Weil curve
$$
E\colon y^2+y=x^3-x^2-10x-20,
$$
  has sign of the functional equation equal to $+1$ and the Mordell-Weil group
   $E(\Q)$ has rank $0$.
 In terms of Dedekind's eta-function, we have that
 $$
 F_E(z)=\eta^2(z)\eta^2(11z)=q-2q^2-q^3+2q^4+q^5+2q^6-2q^7-2q^9-2q^{10}+q^{11}-\dots.
 $$
    We find that the corresponding mock modular form
   $\whZ_E^{+}(z)$ is
\begin{equation*}
\whZ^{+}_E(z)=q^{-1} + 1 + 0.9520... q + 1.5479... q^2 + 0.3493... q^3 +1.9760... q^4 - 2.6095... q^5 + O(q^6).
\end{equation*}
The apparent transcendence of these coefficients arise from  $S(\Lambda_E)=0.381246\dots$. We find that $\Omega_{11}(F_E)=0.2538418...$ which is $1/5$ of the real period of $E$. This $1/5$ is related to the fact that the Mordell-Weil group has a cyclic torsion subgroup of order $5$. A short calculation shows that the expansion of $\mathfrak Z_E(z)$ at the cusp zero is given by
\[\whZ_E^{+}(z)|_0\begin{pmatrix}0&-1\\11&0\end{pmatrix} =\whZ_E^{+}(z)|U(11)+\frac{12}{5}.\]
In particular, the constant term is 17/5.

 We see that $p=5$ is ordinary for $X_0(11)$. Here we illustrate Theorem~\ref{padicformula}.  As a $5$-adic expansion we have that
 $$
 \mathfrak{S}_E(5)=4+2\cdot5^2+4\cdot5^3+\dots
 $$
 which can be thought of as a $5$-adic expansion of $S(\Lambda_E)$ given above. It turns out that
$$\lim_{n\rightarrow +\infty} \dfrac{\left[ q\frac{d}{dq} \zeta (\Lambda_E;\mathcal E_E(z))\,\right]|T(5^n)}{a_E(5^n)}= \mathfrak{S}_E(5)F_E(z)
$$
as a $5$-adic limit.
To illustrate this phenomenon, we let
$$
T_n(E,z):= \dfrac{\left[ q\frac{d}{dq} \zeta (\Lambda_E;\mathcal E_E(z))\,\right]|T(5^n)}{a_E(5^n)}.
$$
We then have that
$$\begin{array}{lll}
T_1(E,z)-4F_E(z)&=
 -5q^{-5} - \frac{50}{3}q - \frac{65}{3}q^{2} + \dots&\equiv 0\pmod{5}\\ \ \ \\
T_2(E,z)-(4+0\cdot 5)F_E(z)&=
\frac{25}{4}q^{-25} - \frac{25}{6}q +\frac{925}{3}q^2-\dots&\equiv 0\pmod{5^2}\\ \ \ \\ \ \ &\ \vdots \\
T_4(E,z)-(4+2\cdot5^2+4\cdot 5^3)F_E(z)&=
-\frac{625}{11}q^{-625} +\frac{5^4\cdot 61301717918}{33}q+\dots&\equiv 0\pmod{5^4}.
\end{array}$$

\end{example}

\begin{example} Here we illustrate Theorem~\ref{thm3}
using the following numerical example computed by  Str\"omberg  \cite{BrStr}.
We consider the elliptic curve $37a1$ given by the Weierstrass model
$$
E: \ \  y^2+y=x^3-x.
$$
The sign of the functional equation of $L(E,s)$ is $-1$, and $E(\Q)$ has rank 1.
The $q$-expansion of $F_E(z)$ begins with the terms
\[
F_E(z)=q - 2q^2 - 3q^3 + 2q^4 - 2q^5 + 6q^6 - q^7 + 6q^9 + 4q^{10} - 5q^{11} + \cdots \in S^{new}_{2}\left(\Gamma_0(37)\right).
\]
Using Remark 3,
we find that the corresponding mock modular form is
\begin{equation*}
\whZ^{+}_E(z)=q^{-1} + 1+ 2.1132...q + 2.3867...q^2 + 4.2201...q^3 + 5.5566...q^4 + 8.3547...q^5 +O(q^6).
\end{equation*}
It turns out that the weight 1/2 harmonic Maass form $f_E(z)=\mathcal{I}_{-3}(\tau,\whZ_E(z))$  corresponds to the Poincar\'e series $\calM_{-3/148,21}$ (see Section \ref{sec:pc})).
Using {\tt Sage} \cite{sage}, Str\"omberg and Bruinier computed all values of $L'(E_d,1)$ for fundamental discriminants $d>0$ such that $\left(\frac{d}{37}\right)=1$ and $|d|\le 15000$. The following table illustrates Theorem~\ref{thm3}.

 \begin{table}[h]
\begin{tabular}{|r|ll|ll|lc|}
\hline %\rule[-3mm]{0mm}{8mm}{8mm}
$d$ && $c^+(d)$ && $L'(E_{d},1)$ && $\rk(E_{d}(\Q))$ \\
\hline% \rule[-3mm]{0mm}{8mm}
$1$  && $-0.2817617849\dots$ && $0.3059997738\dots$ && $1$ \\%26731$\\
$12$  && $-0.4885272382\dots$  && $4.2986147986\dots$&& $1$\\%69658$\\
$21$  && $-0.1727392572\dots$  && $9.0023868003\dots$&& 1\\%04156$\\
$28$ && $-0.6781939953\dots$ && $ 4.3272602496\dots$ && 1\\
$33$ &&$\ \ \, 0.5663023201\dots$ && $3.6219567911\dots$&& 1 \\%73346$\\
%$-40$ && $-0.577676447138672293721\dots$ \\%10457$\\
$\ \ \ \ \ \vdots$ && $\ \ \ \ \ \vdots$&&$\ \ \ \ \ \vdots$ && $\vdots$\\
$1489$&& $\ \ \ \  \ 9$  && $\ \ \ \ \ 0$&& 3\\%50497$\\
$\ \ \ \ \  \vdots$&&$\ \ \ \ \  \vdots$&&$\ \ \ \ \ \vdots$&& $\vdots  $\\
$4393$&& $\ \,  \ \ \ 66$ && $\ \ \ \ \ 0$&& 3 \\ %1.8634999382929359284261114E-30$\\
\hline
\end{tabular}
\end{table}
Stephan Ehlen numerically confirmed that $c^+(d)=   \frac{1}{2\sqrt{d}}\left(\tr^+_{-3}(\whZ_E(z);d)-\tr^-_{-3}(\whZ_E(z);d)\right) $ using {\tt Sage} \cite{sage}.

\end{example}

\begin{example}\label{Zagier}
 In \cite{Za2} Zagier defines the generating functions for the twisted traces of the modular invariant. For coprime fundamental discriminants $d<0$ and $D>1$, he sets
  \[
   f_d=q^{-d}+\sum_{D>0} \left(\frac{1}{\sqrt{D}} \sum_{Q\in \mathcal{Q}_{dD}\setminus \G}\chi(Q)j(\alpha_Q)\right)q^D,
  \]
 where $\mathcal{Q}_{dD}$ are the quadratic forms of discriminant $dD$, $\chi(Q)=\left(\frac{D}{p}\right)$, where $p$ is a prime represented by $Q$ and $\alpha_Q$ is the corresponding CM-point.

 With $d=-\Delta$ and $D=m$ we rediscover a vector-valued version of his results. For example
 \[
  \mathcal{I}_{-3}(\tau,j-744)=f_3= q^{-3}-248 q+26752q^4-85995q^5+1707264q^8-4096248q^9+\cdots.
 \]

\end{example}


\begin{thebibliography}{BrStr}


\bibitem{Pocket} \emph{M. Abramowitz and I.A. Stegun},
Handbook of mathematical functions with formulas, graphs, and mathematical tables.
National Bureau of Standards Applied Mathematics Series, 55 For sale by the Superintendent of Documents, U.S. Government Printing Office, Washington, D.C. 1964 xiv+1046 pp.

\bibitem{Alfes} \emph{C. Alfes}, Formulas for the coefficients of half-integral weight harmonic Maass  forms, Math. Zeitschrift \textbf{227} (2014), 769--795.

\bibitem{Claudia} \emph{C. Alfes and M. Schwagenscheidt}, On a certain theta lifting related to the Shintani lifting, in preparation.

\bibitem{AE} \emph{C. Alfes and S. Ehlen}, Twisted traces of CM values of weak Maass forms,
J. Number Theory \textbf{133} (2013), no. 6, 1827--1845.



\bibitem{ARS} \emph{A. Agashe, K. Ribet, and W. Stein},
The modular degree, congruence primes and multiplicity one, Number Theory, Analysis and Geometry, In memory of Serge Lang, Springer Verlag, Berlin, 2012, 19--50.

\bibitem{AtkinLehnerPaper} A.O.L. Atkin and J. Lehner, {\it Hecke operators on $\Gamma_0(m)$}, Math. Ann. {\bf 185} (1970), 134--160.

\bibitem{BSD1} \emph{B. J. Birch and H. P. F. Swinnerton-Dyer}, Notes on elliptic curves, I, J. Reine Angew. Math. \textbf{212}, (1963), 7--25.

\bibitem{BSD2} \emph{B. J. Birch and H. P. F. Swinnerton-Dyer}, Notes on elliptic curves, II, J. Reine Angew. Math. \textbf{218}, (1965), 79--108.

\bibitem{Bo1}
\emph{R. Borcherds}, Automorphic forms with singularities on
Grassmannians, Invent. Math. \textbf{132} (1998), 491--562.

%\bibitem[Bo2]{Bo2} {\em R. E. Borcherds}, The Gross-Kohnen-Zagier
%  theorem in higher dimensions, Duke Math. J. {\bf 97} (1999),
%  219--233.

%\bibitem[Bo3]{Bo3}
%{\emph R. E. Borcherds}, Correction to ``The Gross-Kohnen-Zagier
%  theorem in higher dimensions'', Duke Math. J. {\bf 105} (2000), 183--184.

%\bibitem{BringmannGuerzhoyKane}
%\emph{K. Bringmann, P. Guerzhoy and B. Kane},
%modular forms as $p$-adic modular forms,
%Trans. Amer. Math. Soc. {\bf 364} (2012), 2393-2410.

\bibitem{BriKaVia} \emph{K. Bringmann, B. Kane and M. Viazovska}, Mock Theta lifts and local Maass forms, Mathematical Research Letters, accepted for publication.


\bibitem{BO1} \emph{K. Bringmann and K. Ono},
The $f(q)$ mock theta function conjecture and partition ranks,
Invent. Math. \textbf{165} (2006), 243--266.


\bibitem{BO2} \emph{K. Bringmann and K. Ono},
Dyson's ranks and Maass forms, Ann. of Math., {\bf 171} (2010), 419--449.
publication.

\bibitem{BO3} \emph{K. Bringmann and K. Ono}, Arithmetic properties of coefficients of half-integral weight Maass-Poincare
series, Math. Ann., {\bf 337} (2007), 591--612.

\bibitem{Br} \emph{J. H. Bruinier}, Borcherds products on
  $\Orth(2,l)$ and Chern classes of Heegner divisors, Springer Lecture
  Notes in Mathematics {\bf 1780}, Springer-Verlag (2002).

\bibitem{BruinierPeriods} \emph{J. H. Bruinier}, Harmonic Maass forms and periods,
Math. Ann. {\bf 357} (2013), 1363--1387.

\bibitem{BF} {\em J. H. Bruinier and J. Funke}, On two geometric theta lifts,
Duke Math. J. {\bf 125} (2004), 45--90.

\bibitem{BrFu06} {\em J. H. Bruinier and J. Funke}, Traces of CM values of modular functions, J. Reine Angew. Math., {\bf 594}. 1--33.

\bibitem{BIF}{\em J. H. Bruinier, J. Funke, and \"O. Imamo\=glu}, Regularized theta liftings and periods of modular functions, J. Reine Angew. Math., accepted for publication. 


\bibitem{BJO} {\em J. H. Bruinier, P. Jenkins, and K. Ono}, Hilbert class polynomials and traces of singular moduli, Math.
Ann. {\bf 334} (2006), 373--393.

\bibitem{BrdGZa08} {\em J. H. Bruinier, G. van der Geer, G. Harder, and D. Zagier}, The 1-2-3 of modular forms,
Springer-Verlag, 2008.
%\texttt{http://arxiv.org/abs/math.NT/0212286}

%\bibitem{BBK} {\em J. H. Bruinier, J. Burgos, and U. K\"uhn},
%Borcherds products and arithmetic intersection theory on Hilbert
%modular surfaces, Duke Math. J. {\bf 139} (2007), 1--88.
% \texttt{math.NT/0310201}

\bibitem{BORh} {\em J. H. Bruinier, K. Ono and R. Rhoades},
Differential operators for harmonic weak Maass forms and the
vanishing of Hecke eigenvalues, Math. Ann. {\bf 342} (2008), 673--693.

\bibitem{BruinierOno} {\em J. H. Bruinier and K. Ono},
Heegner divisors, $L$-functions, and harmonic weak Maass forms, Ann. Math.
{\bf 172} (2010), 2135--2181.

\bibitem{BruinierOno2} {\em J. H. Bruinier and K. Ono}, Algebraic formulas for the coefficients of half-integral weight harmonic weak Maass forms, Adv. Math.,
    {\bf 246} (2013), 198--219.

%\bibitem{BrSt} {\em J. H. Bruinier and O. Stein}, The Weil representation and Hecke %operators for vector-valued modular forms, Mathematische Zeitschrift,
%{\bf 264}  (2010), 249-270.
%\texttt{arXiv.org/abs/0704.1868/}

\bibitem{BrStr} {\em J. H. Bruinier and F. Str\"omberg}, Computation of harmonic weak Maass forms,  Experimental Mathematics {\bf21}(2) (2012), 117--131.

%\bibitem{BKK} {\em J. Burgos, J. Kramer, and U. K\"uhn},
%Cohomological Arithmetic Chow groups,
%J. Inst. Math. Jussieu. {\bf 6} (2007), 1--178.
%\texttt{math.AG/0404122}

%\bibitem{EZ} \emph{M. Eichler and D. Zagier}, The Theory of Jacobi
%  Forms, Progress in Math. {\bf 55}, Birkh\"auser (1985).

\bibitem{Candelori} \emph{L. Candelori}, Harmonic weak Maass forms: a geometric approach,
Math. Ann., accepted for publication.

\bibitem{Cox} {\emph D. A. Cox}, Primes of the form $x^2 + ny^2$, Wiley and Sons, New York, 1989.

\bibitem{Cremona} \emph{J. E. Cremona}, Computing the degree of the modular parameterization of elliptic curves, Math. Comp. \textbf{64} (1995), no. 211, 1235--1250

\bibitem{Damerell1} \emph{R. M. Damerell}, $L$-functions of elliptic curves with
complex multiplication, I, Acta Arith. \textbf{17} (1970), 287--301.


\bibitem{Damerell2} \emph{R. M. Damerell}, $L$-functions of elliptic curves with
complex multiplication, II, Acta Arith. \textbf{19} (1971), 311--317.


%\bibitem{Goldfeld}
%\emph{D. Goldfeld}, Conjectures on elliptic curves over quadratic
%fields, Number Theory, Carbondale, Springer Lect. Notes {\bf 751}
%(1979), 108--118.

\bibitem{Duke} \emph{W. Duke}, Modular functions and the uniform distribution of CM points, Math. Ann. {\bf 334} (2006), no.
2, 241--252.

\bibitem{DIT1}\emph{W. Duke, \"O. Imamo\={g}lu and \'A. T\'oth}, Cycle integrals of the $j$-function and weakly harmonic modular
forms, Ann. Math. {\bf 173} (2011), 947--981.

\bibitem{DIT2}\emph{W. Duke, \"O. Imamo\={g}lu and \'A. T\'oth}, Real quadratic analogs of traces of singular moduli, Int. Math.
Res. Not. IMRN {\bf No. 13}, (2011),  3082--3094.

\bibitem{DJ} \emph{W. Duke and P. Jenkins}, Integral traces of singular values of weak Maass forms, Algebra and Number
Theory, {\bf 2} (2008),  573--593

\bibitem{tables} {\em A. Erd{\'e}lyi, W. Magnus, F.Oberhettinger,  and
             F.G. Tricomi}, Tables of integral transforms. Vol. I,
McGraw-Hill Book Company, Inc., New York-Toronto-London (1954), xx+391.

%\bibitem{GoldSar} {\em D. Goldfeld and P. Sarnak},
%Sums of Kloosterman sums, Invent. Math. {\bf 71} (1983), 243--250.

%\bibitem{Gr} \emph{P. A. Griffiths}, Introduction to algebraic
%  curves, Amer. Math. Soc., Providence, Rhode Island
%  (1989).

\bibitem{Gro} \emph{B. Gross}, Heegner points on $X\sb 0(N)$.
Modular forms (Durham, 1983),  87--105,
%Ellis Horwood Ser. Math. Appl.: Statist. Oper. Res.,
Horwood, Chichester (1984).

\bibitem{Gross} \emph{B. Gross},
Heights and the special values of $L$-series, Number theory
(Montreal, Que., 1985), 115--187, CMS Conf. Proc., 7, Amer. Math.
Soc., Providence, RI, 1987.

\bibitem{GZ} {\em B. Gross and D. Zagier}, Heegner points and derivatives of
L-series, Invent. Math. {\bf 84} (1986), 225--320.

\bibitem{GKZ} \emph{B. Gross, W. Kohnen, and D. Zagier}, Heegner
  points and derivatives of $L$-series. II.  Math. Ann.  {\bf 278}
  (1987), 497--562.

\bibitem{Guerzhoy1} {\em P. Guerzhoy}, A mixed mock modular solution of the $KZ$ equation, Ramanujan J., accepted for publication.

\bibitem{Guerzhoy2} {\em P. Guerzhoy}, Zagier's adele, Res. Math. Sci., accepted for publication.

\bibitem{GKO} {\em P. Guerzhoy, Z. Kent and K. Ono},
$p$-adic coupling of mock modular forms and shadows,
Proc. Natl. Acad. Sci., USA, {\bf 107}, no 14 (2010),
6169--6174.


%\bibitem{Hida} {\emph H. Hida}, {\bf blue book???}.


\bibitem{Hoevel} {\em M. H\"ovel}, Automorphe Formen mit Singularit\"aten auf dem hyperbolischen Raum, TU Darmstadt Diss., 2012.

\bibitem{KS} \emph{S. Katok and P. Sarnak},
Heegner points, cycles and Maass forms, Israel J. Math. {\bf 84}
(1993), 193--227.

%\bibitem{Ka} \emph{Y. Kawai},
%Borcherds products for higher level modular forms, J. Math. Kyoto
%Univ. {\bf 46} (2006), 415--438.

\bibitem{K} \emph{W. Kohnen}, Fourier coefficients of modular forms
of half-integral weight.  Math. Ann.  {\bf 271}  (1985), 237--268.

\bibitem{KZ} \emph{W. Kohnen and D. Zagier},
Values of $L$-series of modular forms at the center of the critical
strip.  Invent. Math.  {\bf 64}  (1981), no. 2, 175--198.

\bibitem{Kolyvagin} \emph{V. Kolyvagin}, Finiteness of $E(\Q)$ and
$\Sha_{E/\Q}$ for a subclass of Weil curves (Russian),
Izv. Akad. Nauk. USSR, ser. Matem. {\bf 52} (1988), 522--540.

%\bibitem{langlands} \emph{R. P. Langlands}, Beyond endoscopy,
%Contributions to automorphic forms, geometry, and number theory,
%Johns Hopkins Univ. Press, Baltimore, Maryland, 2004, 611--697.

\bibitem{Lang} \emph{S. Lang}, Elliptic functions, 2nd Edition, Springer-Verlag, New York, 1987.

\bibitem{McG} \emph{W. J. McGraw}, The rationality of vector-valued modular forms associated with the Weil representation.  Math. Ann.  {\bf 326}
(2003),   105--122.


\bibitem{MillerPixton}\emph{A. Miller and A. Pixton}, Arithmetic traces of non-holomorphic modular invariants, Int. J. of Number
Theory, {\bf 6} (2010),  69--87.

\bibitem{OSk} {\em K. Ono and C. Skinner},
Nonvanishing of quadratic twists of modular L-functions, Invent.
Math. {\bf 134} (1998), 651--660.

\bibitem{O} {\em K. Ono}, Nonvanishing of quadratic twists
of modular $L$-functions with applications for elliptic curves, J.
reine angew. Math. {\bf 533} (2001), 81--97.

\bibitem{O2} {\emph K. Ono}, Unearthing the visions of a master:
harmonic Maass forms in number theory, Proceedings of the 2008 Harvard-MIT Current Developments in Mathematics Conference, Intl. Press,
Somerville, 2009, 347--454.

\bibitem{O3} {\emph K. Ono}, The last words of a genius,
Notices Amer. Math. Soc. {\bf 57} (2010), 1410--1419.

%\bibitem{Ogg} {\em A. P. Ogg}, Rational points on certain elliptic
%modular curves.  Analytic number theory (Proc. Sympos. Pure Math.,
%Vol XXIV, St. Louis Univ., St. Louis, Mo., 1972), 221--231. Amer.
%Math. Soc., Providence, R.I. (1973).

%\bibitem{PP} {\em A. Perelli and J. Pomykala}, Averages
%of twisted elliptic $L$-functions, Acta Arith. {\bf 80} (1997),
%149--163.

%\bibitem{Pribitkin} {\em W. Pribitkin}, A generalization of the
%$Goldfeld-Sarnak estimate on Selberg's Kloosterman zeta-function,
%Forum Math. {\bf 12} (2000), 449--459.
\bibitem{RhoadesPoincare} R. Rhoades, {\it Linear relations among Poincare series via harmonic weak Maass forms},
Ramanujan J. (Ramanujan's 125th birthday issue) {\bf 29} (2012), no. 1-3, 311--320. 


\bibitem{sage}
W.\thinspace{}A. Stein et~al., \emph{{S}age {M}athematics {S}oftware ({V}ersion
  6.2.beta1)}, The Sage Development Team, 2014, {\tt http://www.sagemath.org}.

\bibitem{Sch} {\em A. J. Scholl}, Fourier coefficients of Eisenstein series
 on noncongruence subgroups, Math. Proc. Camb. Phil. Soc.
  {\bf 99} (1986), 11--17.

%\bibitem{Setzer} {\em B. Setzer}, Elliptic curves of prime conductor,
%J. London Math. Soc. {\bf 10} (1975), 367--378.

\bibitem{Shimura} \emph{G. Shimura}. Introduction to the arithmetic theory of
automorphic functions (reprint of the 1971 original), Publ. Math. Soc. Japan, 11,
Kan\^o Memorial Lectures, Princeton Univ. Press, Princeton, 1994.

\bibitem{Sh}
\emph{G. Shimura}, On modular forms of half integral weight. Ann. of
Math. (2) {\bf 97} (1973), 440--481.


%\bibitem{Sk1} \emph{N.-P. Skoruppa}, Developments in the theory
%  of Jacobi forms. In: Proceedings of the conference on automorphic
%  funtions and their applications, Chabarovsk (eds.: N. Kuznetsov and
%  V. Bykovsky), The USSR Academy of Science (1990), 167--185. (see
%  also MPI-preprint 89-40, Bonn (1989).)

\bibitem{Sk2} \emph{N.-P. Skoruppa}, Explicit formulas for the
  Fourier coefficients of Jacobi and elliptic modular forms.  Invent.
  Math.  {\bf 102} (1990), 501--520.

\bibitem{SZ} \emph{N.-P. Skoruppa and D. Zagier}, Jacobi forms
and a certain space of modular forms,  Invent. Math.  {\bf 94}
(1988), 113--146.

\bibitem{Tunnell}
\emph{J. Tunnell}, A classical Diophantine problem and modular forms
of weight $3/2$. Invent. Math. {\bf 72} (1983), no. 2, 323--334.

%\bibitem{W} \emph{M. Waldschmidt}, Nombers transcendents et
%groupes alg\'ebraiques, Ast\'erisque {\bf 69--70} (1979).

\bibitem{Wa}
\emph{J.-L. Waldspurger}, Sur les coefficients de Fourier des formes
modulaires de poids demi-entier,
% (French) [On the Fourier coefficients of modular forms of half-integral weight]
J. Math. Pures Appl. (9) {\bf 60}  (1981), no. 4, 375--484.

\bibitem{Weil} \emph{A. Weil}, Elliptic functions according to Eisenstein and Kronecker, Springer-Verlag, New York, 1976.

%\bibitem{Za1}
%\emph{D. Zagier}, $L$-series of elliptic curves, the
%Birch-Swinnerton-Dyer conjecture, and the class number problem of
%Gauss. Notices Amer. Math. Soc. {\bf 31} (1984), 739--743.

\bibitem{ZagierModDegree} \emph{D. Zagier}, Modular parameterizations of elliptic curves, Canadian Math. Bull. \textbf{28} (1985), no. 3,  372--384

\bibitem{Za2}
\emph{D. Zagier}, Traces of singular moduli, \emph{Motives,
Polylogarithms and Hodge Theory}, Part I. International Press
Lecture Series (Eds. F. Bogomolov and L. Katzarkov), International
Press (2002), 211--244.

\bibitem{Za3} \emph{D. Zagier}, Ramanujan's mock theta functions
and their applications [d'apr\`es Zwegers and
Bringmann-Ono], S\'eminaire Bourbaki 60\'eme ann\'ee, 2006-2007,
no. 986.


\bibitem{Z2} \emph{S. P. Zwegers}, {Mock theta functions},
 Ph.D. Thesis, Universiteit Utrecht, 2002.


\end{thebibliography}
\end{document}